\documentclass{ws-ijcga}
\usepackage{header}
\usepackage{ourstyle}

\begin{document}

\maketitle
\pub{Received (received date)}{Revised (revised date)}{Communicated by (Name)}

\begin{abstract}
	We consider the topological and geometric reconstruction of a geodesic subspace
of $\R^N$ both from the \v{C}ech and Vietoris-Rips filtrations on a finite,
Hausdorff-close, Euclidean sample. Our reconstruction technique leverages the
intrinsic length metric induced by the geodesics on the subspace. We consider
the distortion and convexity radius as our sampling parameters for the
reconstruction problem. For a geodesic subspace with finite distortion and
positive convexity radius, we guarantee a correct computation of its homotopy
and homology groups from the sample. This technique provides alternative
sampling conditions to the existing and commonly used conditions based on weak
feature size and $\mu$--reach, and performs better under certain types of
perturbations of the geodesic subspace. For geodesic subspaces of $\R^2$, we
also devise an algorithm to output a homotopy equivalent geometric complex that
has a very small Hausdorff distance to the unknown underlying~space. 
   

\end{abstract}

\section{Introduction}\label{sec:intro}
With the advent of modern sampling technologies, such as GPS, sensors, medical
imaging, etc., Euclidean point-clouds are becoming widely available for
analysis. In the last decade, the problem of reconstructing an (unknown)
Euclidean shape, from a (noisy) sample around it, has received a far and wide
attention both in theoretical and applied literature; see
\cite{amenta1998crust,Dey:2006:CSR:1196751,SMALE,CHAZALSTAB,co-tpbr-08,chazal2009sampling}.
The nature of such a reconstruction attempt can commonly be classified as being
\emph{topological} or \emph{geometric}. A topological reconstruction is usually
attributed to inferring significant topological features---such as homology and
homotopy groups---of the hidden shape of interest. To be more specific, one may
also say {\em homological} reconstruction or {\em homotopy type}
reconstruction. A much stronger paradigm is the {\em geometric} reconstruction,
where one is interested in producing, from the sample, a Euclidean subset that
is homotopy equivalent and geometrically ``close'' (e.g., in Hausdorff distance)
to the underlying shape.

The nature of the problem and the techniques of the solution change depending on
the type of the shape $X$ and the sample $S$ considered, as well as how their
``closeness'' is measured. The most natural distance measure between two
abstract metric spaces is the Gromov-Hausdorff distance, which measures how
``metrically close'' two metric spaces are. The reconstruction of a geodesic
metric space $X$ from another metric space $S$ that is Gromov-Hausdorff close
to $X$ is considered in \cite{latschev_2001,Chazal2015}. For a Euclidean shape
$X$ and a Euclidean sample $S$, however, the sample density is usually
quantified by their Hausdorff distance. For the Hausdorff-type reconstruction of
Euclidean shapes, see \cite{SMALE,co-tpbr-08,chazal2009sampling,CHAZALSTAB}.

In many applications, a point cloud approximates a geodesic subspace (see
\defref{geodesic}) of Euclidean space. Examples include GPS trajectories
sampled around a road-network (modeled as sampling paths in a graph in $\R^2$), earthquake data
sampled around the filamentary trajectory of the shock, or~3D medical imaging.
The intrinsic geodesics of these underlying shapes enjoy a rich geometric structure.
Capturing that structure from the sampled data is the challenge.
The
length metric $d_L$ (see \eqnref{lengthmetric}) turns them into
geodesic subspaces of $\R^N$. In this work, we consider both topological and
geometric reconstruction of a geodesic subspace $X$ of $\R^N$ from a finite
Hausdorff-close Euclidean~sample.

In shape reconstruction, the use of various simplicial complexes built on the
point-clouds is becoming increasingly popular; see for example
\cite{DeSilva:2004,ATTALI2013448,Adams_2019,co-tpbr-08,kim2020homotopy}. The
most common of them are Vietoris-Rips and \v{C}ech complexes. In this work, we
use filtrations of both of them, and we recognize the \bemph{distortion}
$\delta=\delta(X)$ and \bemph{convexity radius} $\rho=\rho(X)$ of $X$ to be
natural sampling parameters when the geodesic subspaces of $\R^N$ are
considered; see \secref{preliminaries} for their formal definitions.

Our homological reconstruction approach is similar to \cite{co-tpbr-08}, which
is based on the weak feature size (\wfs) of the underlying space. However, the
use of partition of unity, for example, in the proof of
\thmref{Cech-persistence} makes our techniques substantially different. The
novelty of this paper is discerned by the introduction of distortion and
convexity radius as sampling parameters, which is not related to the known
sampling parameters such as the reach,~$\mu$--reach or \wfs
\cite{chazal2009sampling, CHAZALSTAB, ATTALI2013448}. These works are based on
an analysis of the gradient flow of the Euclidean distance function to $X$ in
$\R^N$ and its critical points. Our techniques are substantially different from
that and our results apply to a large class of spaces including smooth
submanifolds of~$\R^N$, finite embedded graphs and higher dimensional simplicial
complexes. As an application of our reconstruction technique, we develop in
\secref{geom-recon} a new topological approach for the reconstruction of
embedded graphs.

%%%%%%%%%%%%%%%%%%%%%%%%%%%
\subsection{Review of Related Works}\label{subsec:related}
%%%%%%%%%%%%%%%%%%%%%%%%%%%
This subsection surveys relevant and pivotal results in shape reconstruction
from point clouds using topological methods, and compares them to the results of
this paper. \tabref{compare} presents a list of some of the most related results
alongside the contribution presented in this work. For necessary definitions and
background we refer the reader to \secref{preliminaries}.

\newcommand{\cellA}[1]{\makecell*[t{p{1in}}]{#1}}
\newcommand{\cellB}[1]{\makecell*[t{p{1.24in}}]{#1}}
\newcommand{\cellC}[1]{\makecell*[t{p{.5in}}]{#1}}
\newcommand{\cellD}[1]{\makecell*[t{p{1.9in}}]{#1}}
\newcommand{\cellE}[1]{\makecell*[t{p{1.5in}}]{#1}}

\renewcommand{\arraystretch}{1.2}

%\begin{landscape}
\begin{table}[tbh]
	\caption[Reconstruction Results]{Reconstruction results. Parameters (params.)
		are: weak feature size
		(\wfs), $\mu$-reach ($R$), shorted edge length ($b$), global
                reach~($\xi$),
		smallest turning angle ($\alpha$), distortion ($\delta$), and convexity
		radius ($\rho$).}\label{tab:compare}
	\centering

	{\small
	\begin{tabular}[t]{|p{0.65in}|p{0.8in}|p{.5in}|p{1.5in}|p{1.5in}|}
		\toprule
		{\bf Authors} & {\bf Space} $X$ & {\bf Param.} & {\bf
			Condition on} $S$
		& {\bf Result}\\
		\colrule
		\cellA{Niyogi \\et al.~\cite{SMALE}} &\cellB{manifolds} &\cellC{$\xi$}
		&\cellD{$\eps<\sqrt{\frac{3}{5}}\xi$ and $S\subset X$ \\is
			$\frac{\eps}{2}$-dense} &\cellE{$S^\eps$ deformation \\retracts to~$X$} \\
		\colrule
		\cellA{Chazal, \\Lieutier~\cite{CHAZALSTAB}}
		&\cellB{compact sets}
		&\cellC{\wfs}
		&\cellD{$d_H(X,S)<\eps<\frac{\wfs(X)}{4}$}
		&\cellE{$Im(i_*)\simeq H_*(X^\alpha)$, where\\
			$i:S^\eps\to S^{3\eps}$ and $\alpha$ is \\sufficiently~small}\\
		\colrule
		\cellA{Chazal, \\Oudot~\cite{co-tpbr-08}}
		&\cellB{compact sets}
		&\cellC{\wfs}
		&\cellD{$d_H(X,S)<\eps<\frac{1}{9}\wfs(X)$,\\ $S$ is finite}
		&\cellE{$Im(i_*)\simeq
			H_*(X^\alpha)$, where\\ $i:\Ri_\eps(S)\to\Ri_{4\eps}(S)$, $\alpha$ \\is
			sufficiently small}\\
		\colrule
		\cellA{Attali \\et al.~\cite{ATTALI2013448}}
		&\cellB{compact sets}
		&\cellC{$\mu$-reach $R$}
		&\cellD{$d_H(X,S)\leq\eps<\lambda^{\text{cech}}(\mu)R$}
		&\cellE{$\C_\alpha(S)$ is homotopy
			\\equivalent to \\$X^\eta$ for $\eta\in(0,R)$}\\
		\colrule
		\cellA{Anjaneya \\et al.~\cite{aanjaneya2012metric}}
		&\cellB{abstract metric \\graphs}
		&\cellC{$b,r$}
		&\cellD{$S$ is an \\$(\eps,R)$-approximation,\\
			$\frac{15\eps}{2} < b <
			\min{\left\{\frac{R}{4},\frac{3b-6\eps}{5}\right\}}$}
		&\cellE{homeomorphic graph}\\
		\colrule

		\cellA{Wasser\-man~\\et al.~\cite{lecci2014statistical}} &\cellB{embedded\\
			metric graphs} &\cellC{$\mu$ of each edge,\\ $\xi$, $\alpha$, $b$, $\tau$}
		&\cellD{$S$ is $\frac{\delta}{2}$-dense in \\$X^\alpha$,
			$0<r+\delta<\xi-2\sigma$, \\and $0<\delta <f(b,\alpha,\tau,\xi,\sigma)$}
		&\cellE{isomorphic pseudo-graph}\\
		\colrule
		\cellA{\thmref{Rips-persistence}}
		&\cellB{geodesic spaces}
		&\cellC{$\delta,\rho$}
		&\cellD{$d_H(X,S)<\frac{\eps}{4}<\frac{\rho}{2\delta(3\delta+2)}$}
		&\cellE{$Im(i_*)\simeq
			H_*(X)$, where $i:\Ri_\eps(S)\to\Ri_{\frac{1}{2}(3\delta+1)\eps}(S)$}\\
		\colrule
		\cellA{\thmref{planar-reconstruction}}
		&\cellB{planar \\subspaces}
		&\cellC{$\delta,\rho$}
		&\cellD{$d_H(X,S)<\frac{\eps}{3}<\frac{\rho}{\delta(15\delta+2)}$}
		&\cellE{Hausdorff-close, homotopy\\ equivalent subset}\\
		\botrule
	\end{tabular}
}
\end{table}
%\end{landscape}

\paragraph{Reach.}
The most well-behaved spaces are smooth Euclidean submanifolds, more generally
spaces with a positive \emph{reach} $\mathsf{r}(X)$. In \cite{SMALE}, the
authors apply geometric and topological tools to reconstruct a smooth
submanifold by the union of Euclidean balls of sufficiently small radius around
a dense subset. The work uses the reach of the embedded submanifold as the
sampling parameter. In a more recent work (\cite{Kim2019}), the authors improve
some of the previously known bounds and develop homotopy-type reconstruction of
a Euclidean (compact) subset with positive reach (and $\mu$-reach) using
\v{C}ech and Vietoris-Rips complexes on a sample.

The above results do not apply when considering shapes beyond the class of
Euclidean submanifolds or spaces that do not have a positive reach, although
such shapes are frequently encountered in practical applications. A common
reason for a space to have a vanishing reach is the presence of sharp corners
and branchings. Such spaces include graphs, embedded simplicial complexes,
manifolds with corners---also the type of shapes we consider in this work for
reconstruction. For manifold reconstruction by Vietoris-Rips complexes in a
slightly different but related context, see
\cite{Adams_2019,adamaszek_homotopy_2016}.

\paragraph{Weak Feature Size, $\mu$-Reach.}
In developing a sampling theory for general compact sets in~$\R^N$, the notion
of \emph{weak feature size} (\wfs) was introduced in \cite{CHAZALSTAB} as the
infimum of the positive critical values of the distance function to the compact
set. Using the \wfs as a sampling condition, the authors developed a
\emph{persistence-based} approach to reconstruct the homology groups and the
fundamental group of a hidden shape from the Euclidean thickenings of the sample
around it.

The results have been further extended in \cite{co-tpbr-08} to facilitate
reconstruction of homology groups from \v{C}ech, Vietoris-Rips, and witness
complexes built on the sample. In comparison with the manifold reconstruction
result in \cite{SMALE}, the techniques of \cite{CHAZALSTAB,co-tpbr-08} apply to
much less regular subspaces of $\R^N$, such as compact Euclidean neighborhood
retracts \cite{Borsuk:1967, Dold:1995}---as long as they have a positive $\wfs$.

The notion of the \wfs of a Euclidean compact set was generalized in
\cite{chazal2009sampling} by introducing the concept of $\mu$-reach, denoted
$\mathsf{r}_\mu(X)$. A homotopy-type reconstruction of spaces with
positive~$\mu$-reach has been developed in \cite{chazal2009sampling,ATTALI2013448}.
Although these works consider for reconstruction spaces beyond the class of
positive \wfs, the difficulty lies in applying the results to shapes as simple
as an embedded tree. Also, choosing a suitable $\mu$ so that the $\mu$-reach is
positive is not always clear.
\begin{figure}[thb]
	\centering
	\includegraphics{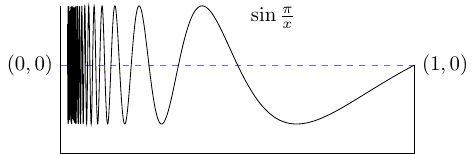}
	\caption{The compact set $X$ (\emph{Warsaw circle}) has a positive \wfs, but
		$X$ and $X^\lambda$ do not have the same homotopy type for any
		$\lambda>0$. In fact, $X$ has the weak homotopy type of a point, whereas
		$X^\lambda$ has the homotopy type of $S^1$. }\label{fig:thickening}
\end{figure}

Our topological reconstruction results (\thmref{Rips-persistence} and
\thmref{Cech-persistence}), are very similar in style to the results presented
in \cite{co-tpbr-08}. However, the use of partition of unity for \v{C}ech
complexes and homotopy equivalence result of Hausmann (\cite{hausmann_1995}) for
Vietoris-Rips complexes make our proofs very different. The \wfs-based
technique employed in \cite{CHAZALSTAB,co-tpbr-08} restricts their results to
work for homology with coefficients only in a field. Moreover, it's not
apparently clear whether the results can easily be extended to higher homotopy
groups. Our reconstruction results, however, do not suffer such restrictions; see
Remark \ref{rem:functor}.

Apart from the fact that we employ $\delta(X)$ and $\rho(X)$ for our sampling
condition, all \wfs (and $\mu$-reach) based results guarantee a reconstruction
of a thickening $X^\lambda$ of $X$ and not $X$ directly.
There are known pathological examples of spaces where the
thickening (however small) is not homotopy equivalent to the underlying space,
such as the {\em Warsaw circle} shown in Figure \ref{fig:thickening}. Although the
homological reconstruction results in our work concern the homological
reconstruction of the subspace~$X$ itself, not the thickening of $X$, they are
not strong enough to apply in the case of the Warsaw circle because of
$\delta(X)=+\infty$ in this case.

Another notable difference in the previously discussed approaches appears in the
cases where $X$ is ``slightly perturbed'', e.g., a submanifold with corners.
Such a perturbation is illustrated in Figure \ref{fig:zero-wfs} for a circle $X$
topologically embedded\footnote{a topological embedding is simply a
$C^0$--embedding.} in $\R^2$. The top part of the space $X$ is the graph of a
rectifiable curve~$\gamma:[0,1]\to \R^2$ such that, when restricted to the
segment $\left[\frac{1}{n+1},\frac{1}{n}\right]$, it is a half-circle of
diameter $\frac{1}{n(n+1)}$ for $n$ odd and a line-segment for $n$ even. For
this space, the set of critical points of the distance function is an infinite
set with an accumulation point at~$(0,0)$. Consequently, $\wfs(X)=0$. However,
$X$ has a finite distortion $\delta=\frac{\pi}{2}$ and a positive intrinsic
convexity radius: $\rho(X)>0$.
 \begin{figure}[thb]
 	\centering
 	\includegraphics[width=4in]{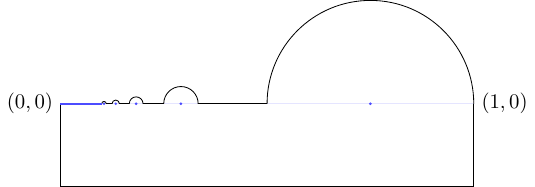}
 	\caption{The space $X$ is a compact Euclidean subspace with $\wfs(X)=0$ and
 		$\mathsf{r}_\mu(X)=0$. The critical points of the distance function are
 		shown in blue; they accumulate at $(0,0)$. However, $X$ has a finite
 		distortion and a positive convexity radius. \label{fig:zero-wfs}}
\end{figure}
Thus $X$ fails to satisfy the conditions of the reconstruction results of
\cite{CHAZALSTAB,co-tpbr-08}, however our results apply to this case. Another
important point, suggested by the example of Figure \ref{fig:zero-wfs}, is that
any embedded submanifold $X$ in $\R^N$ can be perturbed to a submanifold $X'$,
just by adding a small ``spherical cap'' at any of its points. Such a small
perturbation does not change the distortion and the convexity radius too much,
however can produce very small $\wfs$, because we introduce a critical point of
the distance function at the center of the cap. Small values of $\wfs$
result in large sample sizes needed for the reconstruction.

\paragraph{Metric Graph Reconstruction.}
We finish this introduction with a quick summary of some of the existing works
on reconstruction of embedded metric graphs
(\cite{aanjaneya2012metric,Lecci:2014:SAM:2627435.2697074,Chazal2015}).
In~\cite{aanjaneya2012metric}, the authors consider an abstract metric graph and a
sample that is close to it in Gromov-Hausdorff metric, and reconstruct the
structure of the metric graph along with the metric on it. In a more recent work
\cite{Lecci:2014:SAM:2627435.2697074}, the authors show a statistical treatment
of metric graph reconstruction. They consider an embedded metric graph and a
Euclidean sample around it. The Gromov-Hausdorff proximity used in
\cite{aanjaneya2012metric} is replaced by the density assumption. The algorithm
presented in \cite{aanjaneya2012metric} only reconstructs the connectivity of
the vertices of the underlying metric graph and outputs an isomorphic
pseudo-graph. And lastly, we mention that the
first Betti number of an abstract metric graph is computed by considering the
persistent cycles in the Vietoris-Rips complexes of a sample that is very close
to it, with respect to the Gromov-Hausdorff distance; see \cite[Lemma 6.1]{Chazal2015}. In Gromov-Hausdorff type
reconstruction schemes, a small Gromov-Hausdorff distance between the graph and
the sample guarantees a successful reconstruction. These methods are not a good
choice when embedded graphs in $\R^N$ are considered. For an embedded graph with
the induced length metric and a Euclidean sample around it, the Gromov-Hausdorff
distance is not guaranteed to be made infinitely small, even if a dense enough
sample is taken. Also, most of the above mentioned works may be insufficient to
give a geometrically close embedding for the reconstruction. Whereas our
technique, presented in \secref{geom-recon}, can successfully be used to
reconstruct embedded graphs; see \corref{graphs-reconstruction}.

%%%%%%%%%%%%%%%%%%%%%%%%%%%
\subsection{Our Contribution}
One of the major contributions of this work is to reconstruct geodesic subspaces
of $\R^N$, both topologically and geometrically. In our pursuit, we recognize
distortion and convexity radius as new sampling parameters. These sampling
parameters are very natural properties of geodesic spaces.

In \secref{preliminaries}, along with the other important notions of metric
geometry and algebraic topology that we use throughout this paper, we define
convexity radius and distortion of a geodesic space.

In \secref{topo-recon}, our main topological reconstruction results for a
geodesic subspace $X$ of~$\R^N$ are presented. When the distortion is finite and
the convexity radius is positive, the Vietoris-Rips and \v{C}ech filtrations of
the sample are shown to successfully compute the homology and homotopy groups of
$X$ (\thmref{Rips-persistence} and \thmref{Cech-persistence}).

In \secref{geom-recon}, we consider geometric reconstruction of geodesic
subspaces. We construct a complex on the sample as our geometric reconstruction
of the space of interest. \thmref{fundamental} establishes the isomorphism of
their fundamental groups. As an interesting application in \subsecref{graphs},
we consider the geometric reconstruction of planar subspaces and embedded planar
graphs (\defref{metric-graph}) in particular. In \thmref{planar-reconstruction},
we compute a homotopy equivalent geometric complex in the same ambient space
that is also Hausdorff-close to $X$. Since the sample $S$ can be taken to be
finite, our result gives rise to an efficient algorithm (\algref{graph}) for the
geometric reconstruction of planar embedded graphs.

\section{Notation and Background}\label{sec:preliminaries}\label{sec:background}
In this section, we provide a brief overview of useful notation and classical results from
metric geometry and algebraic topology. For more
detailed and complete treatment, we refer the reader to textbooks on metric
geometry~\cite{burago2001course,gromov1999metric} and algebraic
topology~\cite{Kozlov-book,MUNK,spanier1994algebraic}.

%%%%%%%%%%%%%%%%%%%%%%%%%%%
\subsection{Geodesic Subspaces, Distortion, Convexity Radius}
%%%%%%%%%%%%%%%%%%%%%%%%%%%
We first present relevant definitions from metric geometry.

\paragraph{Geodesic Subspaces (of $\R^N$)}
We start with the unit interval $I:=[0,1] \subset \R$. A continuous function
$\gamma\colon I \to \R^N$ is called a path. We call $T=\{t_i\}_{i=0}^k$
a discretization of $I$
if~$0=t_0 < t_1 < t_2 < \ldots < t_k=1$. We create a piecewise linear path
by using straight line segments to connect $\gamma(t_i)$ with $\gamma(t_{i+1})$
for each $i \in \{0,1,\ldots, k-1\}$.
We often equip~$\R^N$ with the Euclidean, or $L_2$
distance, $d_2 \colon \R^N \times \R^N \to \R$ defined by~$d_2(x,y) :=
\norm{x-y}_2$. Let~$\gamma \colon I \to \R^N$ be a (continuous) path.
The
\emph{length} of $\gamma$ is defined as:
\begin{equation*}
    L(\gamma):=\sup_{T} \sum_{i \in \{1,2, \ldots, |T| \}}
            d_2\left(\gamma(t_i),\gamma(t_{i+1})\right),
\end{equation*}
where the supremum is taken over all finite discretizations of $I$.
Furthermore, the curve~$\gamma$
is called \emph{rectifiable} if $L(\gamma)$ is finite.
For a path-connected subset~$X\subseteq\R^N$, we call the restriction of $d_2$
to  $X$ the \emph{restricted metric} on~$X$. We define the \emph{induced length
metric} or \emph{geodesic metric}, $d_L \colon X \times X \to R$,~by
\begin{equation}\label{eqn:lengthmetric}
    d_L(x,y)=\inf\limits_{\gamma:[0,1]\to X} L(\gamma),
\end{equation}
where the infimum is
taken over all paths $\gamma:I\to X$ such that $\gamma(0)=x$ and $\gamma(1)=y$.
\begin{definition}[Geodesic Subspace]\label{def:geodesic} We call
	$X\subseteq\R^N$ a \emph{geodesic subspace} if between any pair of
	points~$x,y\in X$, there exists a rectifiable path on $X$ starting at $x$
	and ending at~$y$ whose length is~$d_L(x,y)$.
\end{definition}

One example of a geodesic subspace is a connected and compact subset of $\R^N$.
The ``niceness'' of an geodesic subspace is quantified by its distortion, a
concept first introduced by M.~Gromov in the
context of knots on Riemannian manifolds~\cite{gromov1978,gromov1983,gromov1999metric}.
For a geodesic subspace~$X\subseteq\R^N$, we consider the map~$f:(X,d_2)\to(X,d_L)$ induced by the
identiy map on $X$.
The distortion of $X$ is the best Lipschitz constant for $f$. More
formally, we have the following definition.
\begin{definition}[Distortion]
    The \emph{distortion} of the induced length metric $d_L$ with respect to
    Euclidean distance over a set $X\subseteq\R^N$ is defined as:
    \[\delta:=\delta(X) = \sup_{x\neq y\in X}\frac{d_L(x,y)}{\norm{x-y}_2}.\]
    For simplicity of exposition, we refer to $\delta$ as \emph{the distortion of $X$}.
\end{definition}

Since $d_L$ is the induced length metric,
$\delta$ is bounded below by one and above by $+\infty$.
If~$X$ is a straight line segment, then~$\delta=1$. On the
other extreme, if $X$ is the subspace~$\{ (x,y)\in \R^2 ~|~ x^2=y^3 \}$, then~$\delta=+\infty$.
To see this, consider the limit as $\eps$ approaches zero of the two
points~$(-\eps^{3/2},\eps) \in X$ and $(\eps^{3/2},\eps) \in \X$, getting arbitrarily
close to the cusp point $(0,0)$.
Thus, both the lower and upper bounds on $\delta$ are tight.
For more on distortion, see~\cite{Sullivan:2008}.

\begin{remark}[Equivalence of Topologies]
	Given a metric space $(X,d)$, we can topologize $X$ with metric balls; that
	is, the topology is generated by sets of the form $B_d(x,r):= \{ y \in X~|~
	d(x,y) < r \}$, where $x\in X$ and $r \in \R$. If we assume that $d_L$ has
	finite distortion with respect to $d_2$, then $(X,d_L)$ and $(X,d_2)$ have
	equivalent topologies. The equivalence of the two topologies is a direct
	consequence of the following inequalities for $x,y \in X$:
	\begin{equation}\label{eqn:metric-equivalence}
		\norm{x-y}_2\leq d_L(x,y)\leq\delta\norm{x-y}_2.
	\end{equation}
\end{remark}
\begin{figure}[thb]
	\centering
	\includegraphics{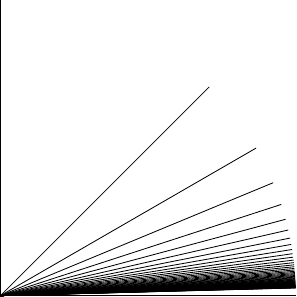}
	\caption{The set $X$, the closure of the union of the falling segments in
		the figure, is known as the \emph{infinite broom}. The topology of
		$(X,d_2)$ is strictly finer than the length metric topology
                of~$(X,d_L)$.  The latter topology is locally path-connected; whereas, the
		former topology is~not.}
	\label{fig:finer-topology}
\end{figure}

Equivalence of the topologies does not generally hold if the distortion of $X$
is not finite. For an example, let $X\subset\mathbb{R}^2$ be the closure of
the union of line segments~$\left\{\left[(0,0),\left(\cos{\frac{\pi}{2i}},
\sin{\frac{\pi}{2i}}\right)\right]\right\}_{i \in \nat}$, as shown in
\figref{finer-topology}. Such a space is also known as the \emph{infinite
broom}. We see that the distortion of the space is infinite by considering the
sequence $a_i=\left(\cos{\frac{\pi}{2i}},\sin{\frac{\pi}{2i}}\right)$ of points
on the right end of the spokes of the~broom:
\[\lim_{i \to \infty} \frac{d_L\big((0,1),a_i\big)}{\norm{(0,1)-a_i}_2} =
\infty.\]

The Euclidean metric topology, in this case, is strictly finer than the length
metric topology, as~$(X,d_L)$ is locally path-connected, but $(X,d_2)$
is not.

\paragraph{Convexity Radius}
Convexity radius of the underlying geodesic subspace is one of the parameters of
$X$ used in all our reconstruction results. We start with its formal definition
from \cite{hausmann_1995}. Although the concept is defined for general length
spaces, we restrict ourselves to only geodesic subspaces.

\begin{definition}[Convexity Radius]\label{def:conv-rad} We define the
	\emph{convexity radius}, denoted $\rho$, of a geodesic
	subspace~$X\subseteq\R^N$ to be the supremum of all $r>0$ such that:
	\begin{enumerate}[(i)]
		\item For all $x,y\in X$ with $d_L(x,y)<2r$, there exists a unique
		(length-minimizing) geodesic path joining $x$ and $y$.

		\item If $x,y,z,u\in X$ such that $d_L(x,y)<r$, $d_L(y,z)<r$,
		$d_L(z,x)<r$, and $u$ is a point on the (length-minimizing) geodesic path
		joining $x$ and $y$,
		then~$d_L(u,z)\leq~\max\left\{d_L(x,z),d_L(y,z)\right\}$.

		\item If $\gamma$ and $\gamma'$ are arc-length parametrized
		(length-minimizing) geodesics on $X$ such that $\gamma(0)=\gamma'(0)$,
		then $d_L\left(\gamma(ts),\gamma'(ts')\right)\leq
		d_L\left(\gamma(s),\gamma'(s')\right)$ for $0\leq s,s'<r$ and $0\leq
		t\leq 1$.
	\end{enumerate}
\end{definition}

Consider a circle in $\R^2$ with perimeter $R$;
its convexity radius is $\frac{R}{4}$.  Also, the convexity radius of
an embedded graph is $\frac{b}{4}$, where $b$ is the length of its smallest
simple cycle. It is well-known that the convexity radius of a compact Riemannian
manifold is positive. The convexity radius of a geodesic space is an intrinsic
property.

%%%%%%%%%%%%%%%%%%%%%%%%%%%
\subsection{Simplicial Complexes, Nerve Lemma}
%%%%%%%%%%%%%%%%%%%%%%%%%%%
We finally conclude this section by outlining a few important notions from
algebraic topology. Readers are referred to
\cite{Kozlov-book,MUNK,spanier1994algebraic} for more details.

\paragraph{Abstract Simplicial Complex}
The combinatorial analogue of a topological space, often used in algebraic and
combinatorial topology, is an abstract simplicial complex. An \emph{abstract
simplicial complex}~$\K$ is a collection of finite sets such that if  $\sigma
\in \K$, then so are all its non-empty subsets.

In general, elements of $\K$ are called \emph{simplices} of $\K$.
The singleton sets in $\K$ are often called the \emph{vertices} of $\K$.
If a simplex~$\sigma\in\K$ has cardinality $(q+1)$,
then it is called a \emph{$q$-simplex} (or the \emph{dimension} of $\sigma$ is $q$ or
$\dim(\sigma)=q$). If~$\sigma'\subseteq\sigma$, then $\sigma'$ is called a
\emph{face} of~$\sigma$.

\paragraph{Simplicial Maps and Contiguity}
Let $\K_1$ and $\K_2$ be abstract simplicial complexes with vertex sets~$\V_1$
and $\V_2$, respectively. A \emph{vertex map} is a map between the vertex
sets. Let~$\phi \colon \V_1 \to \V_2$ be a vertex map.
If, for all $\sigma \in \K_1$, we have~$\phi(\sigma):= \cup_{v\in \sigma} \{\phi(v)\}$ is, in fact, an element
of~$\K_2$, then we say that $\phi$ induces a \emph{simplicial map}
$\phi:\K_1\to\K_2$.
Two simplicial maps~$\phi_1,\phi_2:\K_1\to\K_2$ are called
\emph{contiguous} if for every simplex~$\sigma_1\in\K_1$, there
exists~$\sigma_2\in\K_2$ such that~$\phi_1(\sigma_1)\cup\phi_2(\sigma_1)\subseteq\sigma_2$.
A simplicial map between abstract simplicial complexes is the
combinatorial analogue of a continuous map between topological spaces; likewise,
contiguous simplicial maps play the role of homotopic maps in the combinatorial~world.

\paragraph{Geometric Complex}
Although, abstract simplicial complexes have enough combinatorial structure to
define simplicial homology and homotopy, they are not topological spaces. For an
abstract simplicial complex $\K$ with vertex set $\V$, its \emph{underlying
topological space} or \emph{geometric complex}, denoted as $\mod\K$, is defined
as the space of all functions~$\alpha:\V\to[0,1]$, also called \emph{barycentric
coordinates}, satisfying the following two properties:
\begin{enumerate}[(i)]
	\item $\supp(\alpha):=\{v\in\V\mid\alpha(v)\neq0\}\in\K$
	\item $\sum\limits_{v\in\V}\alpha(v)=1$.
\end{enumerate}
The details on the topologies on $\mod\K$ and their relations can be found
in~\cite{MUNK,spanier1994algebraic}. In this work, we use the standard
\emph{metric topology} on $|\K|$, as defined in \cite{spanier1994algebraic}.
Naturally, a simplicial map $\phi:\K_1\to\K_2$ induces a continuous
map~$\mod\phi:\mod{\K_1}\to\mod{\K_2}$ defined by
\[ \mod\phi(\alpha)(v')=\sum\limits_{\phi(v)=v'}\alpha(v),\text{ for }v'\in\K_2. \] As one
expects, the contiguous simplicial maps induce homotopic continuous maps between
their respective underlying topological spaces; see \cite{spanier1994algebraic}
for a proof.

\paragraph{Nerve Lemma}
A critical ingredient for our \v{C}ech reconstruction results is the Nerve Lemma
or a modification thereof; therefore, we discuss the concept here.  An open
cover $\U=\{U_i\}_{i\in\Lambda}$ of a topological space $X$ is called a
\emph{good cover} if all finite intersections of its elements are contractible.
The \emph{nerve} of~$\U$, denoted $\N(\U)$, is defined to be the simplicial
complex having $\Lambda$ as its vertex set, and for each non-empty $k$-way
intersection $U_{i_1}\cap U_{i_2}\cap\ldots\cap U_{i_k}$, the subset
$\{i_1,i_2,\ldots,i_k\}$ is a simplex of $\N(\U)$. Under the right assumptions,
the nerve preserves the homotopy~type of the union $X$, as stated by the
following fundamental result.
\begin{lemma}[Nerve Lemma~\cite{Alexandroff1928}]
	\label{lem:nerve}
	Let $\U=\{U_i\}_{i\in\Lambda}$ be a good open cover of a topological space
	$X$. Then, the underlying topological space $\mod{\N(\U)}$ is homotopy
	equivalent to $X$.
\end{lemma}

\begin{remark}\label{rem:homotopy-equiv} If the open cover $\U$ is locally
	finite, then the homotopy equivalence in the Nerve
	Lemma is usually constructed with the help of a {\em partition of
        unity} for the cover~\cite{Kozlov-book}.
	Specifically, let $h:X\longrightarrow\mod{\N(\U)}$ be a homotopy
        equivalence. Then, a partition of unity is a collection of
        continuous functions~$\{\varphi_i\colon X\longrightarrow [0,1] \}_{i\in\Lambda}$
        such that for all $x \in X$,
	\begin{equation}\label{eqn:nerve}
		h(x)=\sum_{i\in\Lambda} \varphi_i(x) v_i,
	\end{equation}
	where $v_i$ denotes the vertex of $\N(\U)$ corresponding to the cover element
	$U_i$.  In addition, each~$\varphi_i$ must satisfy the following two requirements:
        (i) for all $i\in\Lambda$, the support of $\varphi_i$,
        denoted~$\supp(\varphi_i)$, is a compact proper subset of
        $U_i$, and
	(ii) for all $x\in X$, $\sum_{i\in\Lambda} \varphi_i(x)=1$.
\end{remark}

\paragraph{\v{C}ech and Vietoris-Rips Complexes}
Consider a subspace $A$ of a metric space $(M,d)$ and a positive scale $\alpha$.
The nerve of the collection of open metric balls of radius $\alpha$
centered at the points of $A$ is
known as the \emph{\v{C}ech complex} of~$A$ at scale (radius) $\alpha$.
We are interested in \v{C}ech complexes in two metric spaces: Euclidean and the
length metric space.  Let $X  \subseteq \R^N$.
Then, the \v{C}ech complex under the standard Euclidean metric is:
$\C_\alpha(X) := \N(\{ \B(x,r) \}_{x\in X})$, where $\B(x,r)$ is the Euclidean
ball of radius $r$ centered at $x$.
The \v{C}ech complex under the length
metric~$(X,d_L)$ is~$\C^L_\alpha(A):= \N(\{
\B^L(x,r) \}_{x\in X}),$ where $\B^L(x,r)$ denotes the metric ball of radius
$r$ centered at $X$ in $(X,d_L)$. Note that
these complexes may be infinite.

The \emph{Vietoris-Rips Complex} is an abstract simplicial complex having a
$k$-simplex for every set of $(k+1)$ points in $A$ of diameter at most $\alpha$.
Explicit knowledge about the entire metric space $(M,d)$ is not needed to
compute the complex. Unlike the \v{C}ech complex, the Vietoris-Rips complex is
completely determined by the restriction of the metric to the subset~$A$.
For~$X\subseteq\R^N$ under the standard Euclidean metric, we denote it simply by
$\Ri_\alpha(X)$. In the case when $A\subseteq X$ equipped with length metric
$(X,d_L)$, we denote the Vietoris-Rips complex by~$\Ri^L_\alpha(A)$.

Together, the definition of convexity radius and Nerve Lemma immediately imply the
following~fact:
\begin{lemma}[\v{C}ech Equivalence]\label{lem:Cech-equivalence}
	Let $X\subseteq\R^N$ be a geodesic subspace with a positive convexity radius
	$\rho$, and let~$0<\eps<\rho$. Let $A$ be an $\eps$-dense
	subset of $X$ with respect to the~$d_L$ metric.
        Then, the complex $\C^L_\eps(A)$ is homotopy equivalent to $X$.
\end{lemma}

\begin{proof}
    Since $A$ is an $\eps$-dense subset of $X$, we know that $\mathcal{U} :=
    \cup_{a \in A}\B^L(a,\eps)$ is an open cover of~$(X,d_L)$. Since $\eps<\rho$
    and by the definition of convexity radius (\defref{conv-rad}), we know that
    for each $x \in X$ and $y \in \B^L(x,\eps)$,  there exists a unique
    length-minimizing geodesic path between~$x$ and $y$.  Using these paths to
    define a deformation retract from~$\B^{L}(x,\eps)$ to~$x$, we conclude that
    the metric balls in $\mathcal{U}$ are contractible. Since any finite
    intersection of metric balls in $\mathcal{U}$ has dimeter less than $2\eps$,
    by the similar argument it is also contractible. Hence, $\mathcal{U}$ is a
    good cover of $X$. By the Nerve Lemma~(\lemref{nerve}), we conclude that the
    complex~$\C^L_\eps(A)$ is homotopy equivalent to $X$.
\end{proof}

\section{Topological Reconstruction}\label{sec:topo-recon}
In this section, we consider the problem of topological reconstruction of a
geodesic subspace~$X$ of $\R^N$ from a noisy sample $S$. From now on, unless
otherwise stated, we assume that the underlying shape $X$ has a positive
convexity radius and a finite distortion, also that the sample~$S$ is a finite
subset of $\R^N$. We show that both \v{C}ech and Vietoris-Rips filtrations of
$S$ can be used to compute the homology and homotopy groups of $X$. Before we
treat each type of complex separately, we show how the \v{C}ech and
Vietoris-Rips complexes behave under Hausdorff perturbation.
\begin{lemma}[Hausdorff Distance and Complexes]
	\label{lem:hausdorff}
	Let $A, B\subseteq\R^N$ be finite, and $\eps$ be a positive number such
	that~$d_H(A,B)<\eps$. Then for any $\alpha>0$, there exist simplicial maps
	\[\C_\alpha(A)\longrightarrow\C_{\alpha+\eps}(B)\]
	and
	\[\Ri_\alpha(A)\longrightarrow\Ri_{\alpha+2\eps}(B)\] induced by a vertex
    map $\xi:A\to B$ such that for every vertex $a\in A$, we have
    $\norm{a-\xi(a)}_2<\eps$.  Moreover, such simplicial maps are unique, up
    to~contiguity.
\end{lemma}

\begin{proof}
	We first note the definition
	\[d_H(A,B)=\inf{\{\eps>0\mid A\subseteq B^\eps, B\subseteq A^\eps\}},\]
	where $A^\eps$ denotes the Euclidean thickening of $A$.

	The definition of Hausdorff distance implies that if $d_H(A,B)<\eps$, there
	exists a (possibly non-unique, non-continuous) map $\xi:A\to B$ such that
	$\norm{a-\xi(a)}_2<\eps$. We show that this vertex map extends to a
	simplicial map between both \v{C}ech and Vietoris-Rips complexes.

	Let $\sigma=\{a_0,a_1,\ldots,a_k\}$ be a $k$-simplex of $\C_\alpha(A)$. By
        definition, there exists a point~$z$ in $\R^N$ such that~$\norm{a_i-z}_2<\alpha$ for
	all $i \in \{0, 1, \ldots, k\}$. By the triangle inequality, we then have
	$$\norm{\xi(a_i)-z}_2\leq\norm{\xi(a_i)-a_i}_2 +
	\norm{a_i-z}_2<\eps+\alpha.$$ So, $\{\xi(a_0),\cdots,\xi(a_k)\}$ is a
	simplex of~$\C_{\alpha+\eps}(B)$. Hence, $\xi$ extends to a simplicial map
	between the \v{C}ech complexes. To argue for the uniqueness of the
	simplicial map, let us assume that~$\eta$ is another simplicial map with the
	property that for every vertex $a\in A$, we have $\norm{a-\eta(a)}_2<\eps$.
	Again from the triangle inequality, we have
	$\norm{\eta(a_i)-z}_2<\eps+\alpha$. So, $\xi(\sigma)\cup\eta(\sigma)$ is a
	simplex of $\C_{\alpha+\eps}(B)$. Hence,~$\xi$ and $\eta$ are contiguous.

	For the Vietoris-Rips complex part, we follow a similar argument. Let
	$\sigma=\{a_0,a_1,\ldots,a_k\}$ be a~$k$-simplex of $\Ri_{\alpha}(A)$. By
	definition, the diameter of $\sigma$ is not greater than $\alpha$. From the
	triangle inequality, we have
	$$\norm{\xi(a_i)-\xi(a_j)}_2\leq\norm{\xi(a_i)-a_i}_2 + \norm{a_i-a_j}_2 +
	\norm{\xi(a_j)-a_j}_2<2\eps+\alpha.$$ So, $\{\xi(a_0),\cdots,\xi(a_k)\}$ is
	a simplex of $\Ri_{\alpha+2\eps}(A)$. Hence, $\xi$ extends to a simplicial
	map also between Vietoris-Rips complexes.
\end{proof}
%

%%%%%%%%%%% Subsection: Rips Filtration %%%%%%%%%%%%%%%%%%%
\subsection{Homology Groups via Vietoris-Rips Complex}\label{subsec:rips-recon}
%%%%%%%%%%%%%%%%%%%%%%%%%%%
We use the following fundamental result from \cite{hausmann_1995} to compute the
homology groups of $X$ from a filtration of Vietoris-Rips complexes on a finite
sample.
\begin{theorem}[Hausmann's Theorem \cite{hausmann_1995}]\label{thm:hausmann}
	Let $X$ be a geodesic subspace with a positive convexity radius~$\rho$. For
	$0<\eps<\rho$, there exists a homotopy equivalence
	$T:\mod{\Ri^L_\eps(X)}\longrightarrow X$.
\end{theorem}
Note that $\Ri^L_\eps(X)$ is usually an infinite Vietoris-Rips complex on the
entire space $X$. A quick corollary of this result is:
\begin{corollary}\label{cor:hausmann}
	Let $X$ be a geodesic subspace with a positive convexity radius $\rho$. For
	$0<\eps'\leq\eps<\rho$, the inclusion
	$i:\Ri^L_{\eps'}(X)\inclusion\Ri^L_\eps(X)$ induces isomorphisms on homology
	and homotopy groups.
\end{corollary}

In order to achieve our result, we use certain simplicial maps to compare
$\Ri^L_*(X)$, $\Ri_*(X)$, and $\Ri_*(S)$.
\begin{lemma}[Euclidean and Intrinsic Rips Complexes]
	\label{lem:rips-intrinsic}
	Let $X$ a geodesic subspace of~ $\R^N$ with a finite distortion $\delta$.
	Then for $A\subseteq X$ and any positive number $\alpha$, we have the
	following simplicial inclusions
	\[\Ri^L_{\alpha}(A)\inclusion\Ri_\alpha(A)
	\inclusion\Ri^L_{\delta\alpha}(A).\]
\end{lemma}

\begin{proof}
	The fact that $\norm{x-y}_2\leq d_L(x,y)$ implies the first inclusion
	$\Ri^L_{\alpha}(A)\inclusion\Ri_{\alpha}(A)$.  Similarly,
	$d_L(x,y)\leq\delta\norm{x-y}_2$ implies the second inclusion.
\end{proof}

\begin{theorem}[Reconstruction via Rips Complex]\label{thm:Rips-persistence} Let
	$X$ be a geodesic subspace of~$\R^N$ with a positive convexity radius $\rho$
	and finite distortion $\delta$. Let $S$ be a finite subset of~$\R^N$, and
	let $\eps$ be a positive number such that
	\[
		d_H(X,S)<\frac{\eps}{4}<\frac{\rho}{2\delta(3\delta+2)}.
	\]
	Then, for any non-negative integer $k$ we have the following isomorphism
	\begin{equation*}\label{eq:Rips-persistence}
		H_k(X)\cong \operatorname{im}\bigl(j_\ast:H_k(\Ri_\eps(S))\inclusion H_k(\Ri_{\frac{1}{2}(3\delta+1)\eps}(S))\bigr)
	\end{equation*}
	where $j_\ast$ is induced by the simplicial inclusion
	$j:\Ri_\eps(S)\inclusion\Ri_{\frac{1}{2}(3\delta+1)\eps}(S)$.
\end{theorem}
\begin{proof}
	We derive the following chain of simplicial maps:
	\begin{equation}\label{eq:maps}
		\Ri^L_{\frac{\eps}{2}}(X)\map[\phi_1] \Ri_{\eps}(S)\map[\phi_2]
		\Ri^L_{\frac{3\eps}{2}\delta}(X)\map[\phi_3]
		\Ri_{(3\delta+1)\frac{\eps}{2}}(S)\map[\phi_4]
		\Ri^L_{\frac{1}{2}(3\delta+2)\delta\eps}(X).
	\end{equation}
	The first map $\phi_1$ is the composition of the simplicial inclusion
	$\Ri^L_{\frac{\eps}{2}}(X)\inclusion\Ri_{\frac{\eps}{2}}(X)$ from
	\lemref{rips-intrinsic} and the simplicial map
	$\rips{X}{\frac{\eps}{2}}\map\Ri_{\eps}(S)$ from \lemref{hausdorff}, thanks
	to the assumption $d_H(S,X)<\frac{\eps}{4}$.

	Now, starting with $\rips{S}{\eps}$ and composing maps from
	\lemref{hausdorff} and \lemref{rips-intrinsic}, respectively, we get the
	second simplicial map $\phi_2$. Similarly, we get the maps $\phi_3$ and
	$\phi_4$.

	From \lemref{hausdorff}, we first note that the composition
	$\phi_3\circ\phi_2$ is contiguous to the inclusion:
	\[j:\Ri_\eps(S)\inclusion\Ri_{(3\delta+1)\frac{\eps}{2}}(S).\]
	Therefore,
	they induce homotopic maps on the respective underlying topological
	spaces. Consequently, we have $(\phi_3\circ\phi_2)_*=j_*$. We first argue
	that~${\phi_2}_*$ is surjective and ${\phi_3}_*$ is injective.

	By the choice of the simplicial maps in \lemref{rips-intrinsic} and
	\lemref{hausdorff}, we observe that~$\phi_2\circ\phi_1$ is contiguous to the
	inclusion
	\[\Ri^L_{\frac{\eps}{2}}(X)\inclusion \Ri^L_{\frac{3\eps}{2}\delta}(X).\]
	By
	\corref{hausmann}, the inclusion induces isomorphism on homology, hence so
	does $\phi_2\circ\phi_1$. In particular,~$(\phi_2\circ\phi_1)_*$ is
	surjective. Hence, we have ${\phi_2}_*$ is surjective, and ${\phi_1}_*$
	is injective.

	Also, $\phi_4\circ\phi_3$ is contiguous to the inclusion
	\[\Ri^L_{\frac{3\eps}{2}\delta}(X)\inclusion
	\Ri^L_{\frac{1}{2}(3\delta+2)\delta\eps}(X),\]
	which induces an isomorphism on
	homologies. Therefore, ${\phi_3}_*$ induces an injective homomorphism.

	Since we have $j_*={\phi_3}_*\circ{\phi_2}_*$ and ${\phi_2}_*$ is
	surjective, the image of $j_*$ is the image of ${\phi_3}_*$. On the other
	hand, we know that $Im({\phi_3}_*)$ is isomorphic to
	$H_*\left(\Ri^L_{\frac{3\eps}{2}\delta}(X)\right)/Ker({\phi_3}_*)$. As we
	have already shown that~${\phi_3}_*$ is injective, its kernel is trivial.
	Therefore, the image of $j_*$ is isomorphic to
	$\Ri^L_{\frac{3\eps}{2}\delta}(X)$. Since $\frac{3\eps}{2}\delta<\rho$,
	\thmref{hausmann} implies that~$\Ri^L_{\frac{3\eps}{2}\delta}(X)$ is, in
	fact, homotopy equivalent to $X$. This completes the proof.
\end{proof}

The Vietoris-Rips reconstruction result works also for an infinite sample $S$. In
applications, however, we are computationally constrained to use only finite
samples.

%%%%%%%%%%% Subsection: Cech Filtration %%%%%%%%%%%%%%%%%%%
\subsection{Homology Groups via \v{C}ech Complex}
%%%%%%%%%%%%%%%%%%%%%%%%%%%
The reconstruction of homology groups via the Vietoris-Rips filtration (see
\thmref{Rips-persistence} in \subsecref{rips-recon}) was due to the homotopy
equivalence theorem (\thmref{hausmann}).  In this subsection, we use
\v{C}ech filtration to obtain similar reconstruction results. The Nerve Lemma
(\lemref{nerve}) is resorted to as the \v{C}ech alternative to \thmref{hausmann}. Like the Vietoris-Rips case, we still use different simplicial maps to
compare~$\C^L_*(X)$,~$\C_*(X)$, and~$\C_*(S)$. The approach involves a
(controlled) variant of the partition of unity; see \lemref{pou}.

\begin{lemma}[Euclidean and Intrinsic \v{C}ech Complexes]
	\label{lem:cech-intrinsic}
	Let $X$ a geodesic subspace of~ $\R^N$ with a finite distortion $\delta$.
	Then for $A\subseteq X$ and any positive number $\alpha$, we have the
	following simplicial inclusions
	\begin{equation*}
		\C^L_\eps(A)\inclusion\C_\alpha(A)
		\inclusion\C^L_{2\delta\alpha}(A).
	\end{equation*}
\end{lemma}

\begin{proof}
	From $\norm{x-y}_2\leq d_L(x,y)$, we have the first inclusion.

	On the other hand, for any $x,y\in X$ we have
	$d_L(x,y)\leq\delta\norm{x-y}_2$. Let $\sigma=\{x_0,...,x_k\}$ be a simplex
	of $\cech{A}{\alpha}$.  Then $\norm{x_i-x_j}_2< 2\alpha$, consequently
	$d_L(x_i,x_j)< 2\delta\alpha$ for all $1\leq i, j\leq k$. This implies
	\[
	\{x_0,x_1,\ldots,x_k\}\subset \bigcap^k_{i=0} \B^L(x_i,2\delta\alpha),
	\]
	where $\B^L(x_i,r)$ denotes the ball of radius $r$ centered at $x_i$ in the
    metric space~$(X,d_L)$. Therefore $\sigma\in\C^L_{2\delta\alpha}(A)$, and
    this verifies the second inclusion.
\end{proof}

We begin with a lemma that is analogous to \corref{hausmann} in the \v{C}ech
regime:
\begin{lemma}[Inclusion of Covers]\label{lem:nerve-inclusion} Let
	$\U=\{U_i\}_{i\in\Lambda}$ and $\U'=\{U'_i\}_{i\in\Lambda}$ be locally-finite,
	good open covers of a para-compact topological space~$X$ such that
	$U_i\subseteq U'_i$ for each $i$. Then, the inclusion
	\[i:\N(\U)\inclusion\N(\U')\]
	induces isomorphisms on the homology and homotopy groups of the respective
	geometric complexes.
\end{lemma}

\begin{proof}
	Consider the following commutative diagram:
	\begin{equation*}
		\begin{tikzpicture}[baseline=(current  bounding  box.center)]
			\node (k1) at (-2,0) {$\mod{\N(\U)}$};
			\node (k2) at (2,0) {$\mod{\N(\U')}$};
			\node (k3) at (0,-2) {$X$};
			\draw[rinclusion] (k1) to node[auto] {$i$} (1,0);
			\draw[map,swap] (k3) to node[auto,swap] {$h$} (k1);
			\draw[map,dashed,swap] (k3) to node[auto] {$i\circ h$} (k2);
		\end{tikzpicture}
	\end{equation*}
	where the map $h=\sum\varphi_i u_i$ is obtained from an arbitrary partition of
	unity $\{\varphi_i\}$ for $\U$. By the Nerve Lemma (\lemref{nerve}), $h$ is a
	homotopy equivalence (\cite{Kozlov-book}). Since $U_i\subseteq U'_i$,
	$\{\varphi_i\}$ is a partition of unity for~$\U'$. So, $i\circ h$ is also a
	homotopy equivalence. Since the maps $h$ is a homotopy equivalence, we
	conclude that $i$ induces an isomorphism on homology and homotopy groups.
\end{proof}

We now state the following extension of the partition of unity. Follow
\cite{munkres2000topology} for a proof.
\begin{lemma}[Controlled Partition of Unity]
	\label{lem:pou}
	Let $\{U_i\}$ and $\{V_i\}$ be open covers of a paracompact, Hausdorff space
	$X$ such that $\overline{V_i}\subseteq U_i$ for each $i$.  Then, there exists
	a partition of unity $\{\varphi_i\}$ subordinate to $\{U_i\}$ such that
	$V_i\subseteq\supp{\varphi_i}\subseteq U_i$ for all $i$.
\end{lemma}

We now use the controlled partition of unity to prove the following important
lemma.
\begin{lemma}[Commuting Diagram]\label{lem:commute}
	Let $X, Y$ be paracompact, Hausdorff spaces with a continuous map~$f:X\to
	Y$. Let $\U=\{U_i\}$ and $\V=\{V_i\}$ be good, locally finite, open covers
	of $X$ and $Y$ respectively, such~that
	\begin{enumerate}[(a)]
		\item $\bigcap_i V_i\neq\emptyset$ implies $\bigcap_i U_i\neq\emptyset$, i.e.,
		we have the simplicial inclusion $j:\N(\V)\to\N(\U)$ that sends the vertex
		corresponding to $V_i$ to the vertex corresponding to $U_i$,
		\item $\overline{f^{-1}(V_i)}\subseteq U_i$ for all $i$.
	\end{enumerate}

	Then, the following diagram commutes, up to homotopy:
	\begin{equation*}
		\begin{tikzpicture} [scale=1, baseline=(current  bounding  box.center)]
			\node (k1) at (-2,0) {$\mod{\N(\V)}$};
			\node (k2) at (2,0) {$\mod{\N(\U)}$};
			\node (k3) at (-2,-2) {$Y$};
			\node (k4) at (2,-2) {$X$};
			\draw[rinclusion] (k1) to node[auto] {$j$} (k2);
			\draw[map] (k3) to node[auto] {$h_Y$} (k1);
			\draw[map] (k4) to node[auto,swap] {$h_X$} (k2);
			\draw[map] (k4) to node[auto,swap] {$f$} (k3);
		\end{tikzpicture}
	\end{equation*}
	where $h_X,h_Y$ are homotopy equivalences from \eqnref{nerve}.
\end{lemma}
\begin{proof}
	We make use of the controlled partition of unity lemma to prove our
	result. Let us choose a partition of unity $\{\phi_i\}$ subordinate to
	$\{V_i\}$. One can choose $h_Y$ so that for each~$y\in Y$,
	\[
	h_Y(y)=\sum_i\phi_i(y)v_i,
	\]
	where $v_i$ is the vertex of $\N(\V)$ corresponding to $V_i$.

	Since $\{f^{-1}(V_i)\}$ is an open cover of $X$ with
	$\overline{f^{-1}(V_i)}\subseteq U_i$ for each $i$, by \lemref{pou} we can
	choose a partition of unity $\{\psi_i\}$ subordinate to $\{U_i\}$ such that
	for each $i$
	\[
	f^{-1}(V_i)\subseteq\supp{\psi}_i.
	\]
	Also, choose $h_x$ such that for each $x\in X$
	\[
	h_X(x)=\sum_i\psi_i(x)u_i,
	\]
	where $u_i$ is the vertex of $\N(\U)$ corresponding to $U_i$.

	To see that the diagram commutes, up to homotopy, it suffices to show that $(j
	\circ h_Y \circ f)$ is homotopic to $h_X$. We start with a point $x_0\in X$
	\[(j \circ h_Y \circ f)(x_0) = j\big(\sum_i\phi_i(f(x_0))v_i\big) =
	\sum_i\phi_i(f(x_0))j(v_i) = \sum_i\phi_i(f(x_0))u_i.\]
	On the other hand,
	$h_X(x_0)=\sum_i\psi_i(x_0)u_i$. Now if $\phi_i(f(x_0))$ is non-zero for some
	$i$, then~$f(x_0)\in V_i$, and consequently $x_0\in f^{-1}(V_i)\subseteq
	U_i$. From our choice of the support of $\psi_i$ and $\psi_i(x_0)$ has to be
	non-zero. This shows that both $(j \circ h_Y \circ f)(x_0)$ and $h_X(x_0)$ lie
	in an (open) simplex of $\N(\V)$. Due to convexity of simplices, the following
	(straight-line) homotopy is well-defined:
	\[
	F(x,t) = \sum_i\left[t\psi_i(x) + (1-t)\phi_i(x)\right]u_i.
	\]
	This shows that $(j \circ h_Y \circ f)$ is homotopic to $h_X$.
\end{proof}

Now we are in a position to prove our reconstruction result for \v{C}ech
complexes.

\begin{theorem}[Reconstruction via \v{C}ech complex]\label{thm:Cech-persistence}
	Let $X$ be a geodesic subspace of~$\R^N$ with a positive convexity
	radius~$\rho$ and finite distortion $\delta$. Let $S$ be a finite subset
	of~$\R^N$, and let $\eps$ be a positive number such that
	\[d_H(X,S)<\eps<\frac{\rho}{2\delta(4\delta+1)}.\] Then, any non-negative
	integer $k$ we have the following isomorphism
	\begin{equation}\label{eq:Cech-persistence}
	H_k(X)\cong \operatorname{im}\bigl(j_\ast:H_k(\C_\eps(S))\inclusion H_k(\C_{(4\delta+1)\eps}(S))\bigr)
	\end{equation}
	where $j_\ast$ is induced by the simplicial inclusion
	$j:\C_\eps(S)\longrightarrow\C_{(4\delta+1)\eps}(S)$.
\end{theorem}

\begin{proof}
	We first note from $d_H(X,S)<\eps$ and \lemref{hausdorff} that there is a map
	$\xi:S\to X$ such that for each~$s\in S$,
	\begin{equation}\label{eqn:xi}
		\norm{s-\xi(s)}_2<\eps.
	\end{equation}

	Let $X'=\xi(S)$. Then, \eqnref{xi} implies $d_H(S,X')<\eps$, hence
	$d_H(X,X')<2\eps$ by the triangle inequality.

	We now derive the following chain of simplicial maps:
	\begin{equation*}\label{eqn:chain}
		\C_{\eps}(S)\map[\phi_1] \C^L_{4\eps\delta}(X')\map[\phi_2]
		\C_{(4\delta+1)\eps}(S)\map[\phi_3] \C^L_{2\delta(4\delta+1)\eps}(X').
	\end{equation*}

	The first map $\phi_1$ is the composition of the simplicial map
	$\C_\eps(S)\inclusion\C_{2\eps}(X')$ from \lemref{hausdorff} (due to
	$d_H(S,X')<\eps$) and the simplicial inclusion
	$\C_{2\eps}(X')\inclusion\C^L_{4\delta\eps}(X')$ from \lemref{cech-intrinsic}.

	Similarly, starting with $\C^L_{4\delta\eps}(X')$ and composing maps from
	\lemref{cech-intrinsic} and \lemref{hausdorff}, respectively, we get the
	second simplicial map $\phi_2$. The other map $\phi_3$ is also
	obtained repeating the exact same argument for a different scale as for
	$\phi_1$.

	We first observe that the choice of simplicial maps in \lemref{cech-intrinsic}
	and \lemref{hausdorff} makes~$\phi_2\circ\phi_1$ contiguous to the given
	natural inclusion $j$ of $\C_{\eps}(S)$ into
	$\C_{2\delta(4\delta+1)\eps}(S)$.  We now consider the following~diagram:
	\begin{equation}
		\label{eqn:cech-rectangles}
		\centering
		\begin{tikzpicture} [baseline=(current  bounding  box.center)]
			\node (k1) at (-6,0) {$\mod{\C_{\eps}(S)}$};
			\node (k2) at (-2,0) {$\mod{\C^L_{4\delta\eps}(X')}$};
			\node (k3) at (2,0) {$\mod{\C_{(4\delta+1)\eps}(S)}$};
			\node (k4) at (6,0) {$\mod{\C^L_{2\delta(4\delta+1)\eps}(X')}$};
			\node (k5) at (-6,-2) {$S^\eps$};
			\node (k6) at (-2,-2) {$X$};
			\node (k7) at (6,-2) {$X$};
			\draw[map] (k1) to node[auto] {$\phi_1$} (k2);
			\draw[map] (k2) to node[auto] {$\phi_2$} (k3);
			\draw[map] (k3) to node[auto] {$\phi_3$} (k4);
			\draw[map] (k5) to node[auto] {$h_1$} (k1);
			\draw[map] (k6) to node[auto] {$h_2$} (k2);
			\draw[linclusion] (k6) to node[auto,swap] {$i$} (k5);
			\draw[linclusion] (k7) to node[auto,swap] {$Id$} (k6);
			\draw[map] (k7) to node[auto] {$h_3$}(k4);
		\end{tikzpicture}
	\end{equation}
	To show that the diagram commutes up to homotopy,  we first explain the
	horizontal maps in the bottom row of \eqnref{cech-rectangles}. Since
	$d_H(X,S)<\eps$, we get the first inclusion $X\subseteq S^\eps$.  The three
	vertical maps are homotopy equivalences that come from the Nerve Lemma
	(\lemref{nerve}) for various good open covers as constructed in
	\lemref{commute}. The first vertical map~$h_1$ is obtained for the open
	cover $\U_1=\{\B(x,\eps)\}_{x\in S}$ of $S^\eps$ by Euclidean balls. The
	other two vertical maps,~$h_2$ and $h_3$, are corresponding to the
	(intrinsic) covers~$\U_2$ and~$\U_3$ of~$(X,d_L)$ by the intrinsic balls of
	radii $2\delta\eps$ and $4\delta(2\delta+1)\eps$, respectively. The
	assumption~$4\delta(2\delta+1)\eps<\rho$ implies that they are indeed good
	(intrinsic) covers of $X$. Therefore, by \lemref{Cech-equivalence} we get
	the homotopy equivalences $h_2$ and $h_3$.

	Apply \lemref{commute} to each of the rectangles in \eqnref{cech-rectangles}
	to show that the diagram is homotopy commutative, and therefore it commutes
	on the homology level. The commutativity then implies that $\phi_1$ induces
	a surjective homomorphism and $\phi_2$ induces an injective homomorphism on
	the homology groups. As a consequence,
	$\im(\phi_{2*}\circ\phi_{1*})=\im(\phi_{2*})=H_k(X)$ on the $k$-th homology
	group.  Also, we note that $\phi_2\circ\phi_1$ is homotopic to the given
	simplicial inclusion $j$.

	To see that the first rectangle commutes, we consider the covers $\U_1$ and
	$\U_2$ of $S^\eps$ and~$(X,d_L)$. Note that for any $x\in S$, the choice of
	$\xi(x)$ implies that $i^{-1}(\B(x,\eps))=\B(x,\eps)\cap
	X\subseteq\B^L(\xi(x),2\delta\eps)$. Consequently, $\overline{\B(x,\eps)\cap
		X}\subseteq\B^L(\xi(x),4\delta\eps)$. A similar argument also
	applies to other rectangle. Therefore by \lemref{commute}, the diagram
	\eqnref{cech-rectangles} commutes.
\end{proof}

\begin{remark}\label{rem:functor} We remark that \thmref{Rips-persistence} and
	\thmref{Cech-persistence} of this section can be formulated in terms of any
	natural functor from the category of topological spaces (with continuous
	maps as morphisms) to the category of groups (with group homomorphisms). In
	particular, the results extend immediately to homology groups
	$H_\ast(\,\cdot\, ; G)$ with coefficients in any abelian group $G$, or
	homotopy groups $\pi_\ast(\,\cdot\,)$.
\end{remark}

\section{Geometric Reconstruction}\label{sec:geom-recon}
In the previous section, we used filtrations of both the \v{C}ech and the
Vietoris-Rips complexes to compute the homology and homotopy groups of our
hidden geodesic subspace $X$ from a noisy sample $S$ around it. The results,
however, do not provide us with a topological space that faithfully carries the
topology of~$X$. To remedy this, we consider the problem of geometric
reconstruction of geodesic subspaces.

In \subsecref{fundamental}, we introduce a new metric $d_\eps$ on $S$. As our
first step towards capturing the homotopy type, we show in \thmref{fundamental}
that the Vietoris-Rips complex of $(S,d_\eps)$ and $X$ have isomorphic
fundamental groups.  Finally in \subsecref{graphs}, we further use this complex
for the geometric reconstruction of embedded graphs.

%%%%%%%%%%%%%%%%%%%%%%%%%%%
\subsection{Recovery of the Fundamental Group}
%%%%%%%%%%%%%%%%%%%%%%%%%%%
\label{subsec:fundamental}
For any fixed $\eps>0$, we first consider the Euclidean Vietoris-Rips complex
$\Ri_\eps(S)$ on the sample $S$. Regardless of how dense the sample $S$ is,
$\Ri_\eps(S)$ is not guaranteed to be homotopy equivalent to $X$ in general; as
shown in \figref{recon}. This is not surprising, because the Euclidean metric on
$S$, used to compute the complex, can be very different from the length
metric~$d_L$ on~$X$. Our goal is to approximate $d_L$ by the shortest path
metric, denoted $d_\eps$, on the one-skeleton of $\Ri_\eps(S)$. Let us denote
the one-skeleton of $\Ri_\eps(S)$ by $G_\eps$. Since $\Ri_\eps(S)$ is an
abstract simplicial complex,~$G_\eps$ inherits the structure of an abstract
graph. However, we turn its geometric complex $\mod{G_\eps}$ into a metric graph
by defining the metric $d_\eps$ on it in the following way: the metric, when
restricted to an edge $(s,t)$, is isometric to a real interval of length
$\norm{s-t}_2$.

We show in \lemref{d-eps} that $d_\eps$ nicely approximates the metric $d_L$,
which the Euclidean sample is oblivious to. For any positive scale $\alpha$, we
denote the Vietoris-Rips complex of $S$ in the $d_\eps$ metric
by~$\Ri^\eps_\alpha(S)$. The metric $d_\eps$ can be computed in $O(k^3)$-time
from a sample~$(S,d_2)$ of size~$k$. In the following lemma, we compare
the metric $d_\eps$ with the standard Euclidean metric~$d_2$ and the
length metric $d_L$.

\begin{lemma}[Minimal Covering of Paths]\label{lem:d-eps} Let $X$ be a geodesic
	subspace of $\R^N$. Let~$S$ be a subset of $\R^N$ and~$\eps>0$ such that
	$d_H(X,S)<\frac{\eps}{3}$. For any path $\gamma$ joining any two
	points~$x,y\in X$, we can find a sequence~$\{a_i\}_{i=0}^k\subseteq S$ with
	$\norm{a_{i+1}-a_i}_2<\eps$ such that
	$$\sum\limits_{i=0}^{k-1}\norm{a_{i+1}-a_{i}}_2<3l(\gamma).$$ Moreover, $a_0$
	and $a_k$ can be chosen to be any points with $\norm{x-a_0}_2<\frac{\eps}{3}$
	and $\norm{y-a_K}_2<\frac{\eps}{3}$.
\end{lemma}

\begin{proof}
	Since $d_H(X,S)<\frac{\eps}{3}$, there exists $a_0\in S$ such that
	$\norm{x-a_0}_2<\frac{\eps}{3}$. We now iteratively define the sequence
	$\{a_i\}\subseteq S$, along with a sequence $\{t_i\}_0^k \subset [0,1]$ that
	defines a partition of $[0,1]$. We set $t_0=0$.  Assuming both $a_i$ and
	$t_i$ are defined, we define $t_{i+1}\in [0,1]$ in the following way: if
	$\gamma([t_i,1])\cap\partial{\B\left(a_i,\frac{2\eps}{3}\right)}\neq\emptyset$,
	we set
	\[t_{i+1}=\min\{t\in[t_i,1]\mid\gamma(t)\in
	\partial{\B\left(a_i,\frac{2\eps}{3}\right)}\}.\] Otherwise if
	$\gamma([t_i,1])\cap\partial{\B\left(a_i,\frac{2\eps}{3}\right)}=\emptyset$,
	set $t_{i+1}=1$. Since $d_H(S,X)<\frac{\eps}{3}$, we set $a_{i+1}\in S$ to
	be a point in $S$ such that
	$\norm{\gamma(t_{i+1})-a_{i+1}}_2<\frac{\eps}{3}$. This procedure forces
	$t_{i+1}$ to be strictly greater than $t_i$, hence $\{t_i\}$ defines a
	partition of $[0,1]$.  Therefore,
	\[
	l(\gamma) = \sum_{i=0}^kl(\gamma |_{[t_i,t_{i+1}]})
	\geq\sum_{i=0}^k\norm{\gamma(t_i)-\gamma(t_{i+1})}_2
	\geq\sum_{i=0}^k\frac{\eps}{3} \geq\frac{1}{3}\sum_{i=0}^k\norm{a_{i+1}-a_i}_2.
	\]

	We also note that
	\[0<\norm{a_{i+1}-{a_i}}_2\leq\norm{a_{i+1}-\gamma(t_{i+1})}_2
	+\norm{\gamma(t_{i+1})-a_i}_2<\frac{\eps}{3}+\frac{2\eps}{3}=\eps.\]
\end{proof}

Analogous to \lemref{hausdorff}, we get the following useful simplicial maps.
\begin{lemma}[Vietoris-Rips Inclusion by $d_\eps$]\label{lem:dL-inclusion}
	Let $X$ a geodesic subspace $X\subseteq\R^N$. Let $S\subseteq\R^N$ and
	$\eps>0$ be such that $d_H(X,S)<\frac{\eps}{3}$. For any $\alpha>0$,
	\begin{enumerate}[(i)]
		\item there exists a natural simplicial inclusion
		\[
		\Ri^\eps_{\alpha}(S)\inclusion\Ri_\alpha(S).
		\]
		\item
		there exists a simplicial map
		\[
		\xi:\Ri^L_\alpha(X)\map\Ri^\eps_{3\alpha}(S)
		\]
		induced by the vertex map $\xi$ that sends a vertex $x\in X$ to $s\in S$
		such that $\norm{x-s}_2<\frac{\eps}{3}$.
	\end{enumerate}
\end{lemma}

\begin{proof}
	\begin{enumerate}[(i)]
		\item Follows immediately from the definition of the metric $d_\eps$.

		\item
		As observed before in \lemref{hausdorff}, the assumption
		$d_H(X,S)<\frac{\eps}{3}$ ensures that there is a vertex map $\xi:X\to S$
		such that for each $x\in X$ we have $\norm{x-\xi(x)}_2<\frac{\eps}{3}$.

		We show that the map extends to a simplicial map. Let
		$\sigma=\{x_0,x_1,\cdots,x_k\}$ be a $k$-simplex of $\Ri^L_\alpha(X)$. Then,
		$d_L(x_i,x_j)\leq\alpha~\forall i,j$. Now by \lemref{d-eps}, there exists a
		path joining $\xi(x_i)$ and $\xi(x_j)$ in $G_\eps$, moreover
		$d_\eps(\xi(x_i),\xi(x_j))<3\alpha$. So, $\xi(\sigma)$ is a simplex of
		$\Ri^\eps_{3\alpha}(S)$. Hence, the vertex map extends to a simplicial map.
	\end{enumerate}
\end{proof}

We now show that the fundamental group of the Vietoris-Rips complex on $S$ under
the metric $d_\eps$ is isomorphic to that of $X$. We tolerate the sloppiness
from ignoring the basepoint.
\begin{theorem}[Fundamental Group]\label{thm:fundamental}
	Let $X$ be a connected geodesic subspace of $\R^N$ with a positive convexity
	radius $\rho$ and a finite distortion $\delta$. Let $S\subseteq\R^N$ and
	$\eps>0$ be such that
	\[
	 d_H(X,S)<\frac{\eps}{3}<\frac{\rho}{\delta(15\delta+2)}.
	 \]
	 Then, the
	fundamental groups of $\Ri^\eps_{5\eps\delta}(S)$ and $X$ are isomorphic.
\end{theorem}

\begin{proof}
	We derive the following chain of simplicial maps:
	\begin{equation*}
		\Ri_{\eps}(S)\map[\phi_1] \Ri^L_{\frac{5\eps\delta}{3}}(X)\map[\phi_2]
		\Ri^\eps_{5\delta\eps}(S)\inclusion[\phi_3] \Ri_{5\delta\eps}(S)\map[\phi_4]
		\Ri^L_{\delta(15\delta+2)\eps/3}(X).
	\end{equation*}
	The map $\phi_1$ is the composition of the simplicial map
	$\Ri_\eps(S)\map\Ri_{\frac{5\eps}{3}}(X)$ from \lemref{hausdorff} and the
	simplicial inclusion
	$\rips{X}{\frac{5\eps}{3}}\inclusion\Ri^L_{\frac{5\eps\delta}{3}}(S)$ from
	\lemref{rips-intrinsic}, thanks to the assumption
	$d_H(S,X)<\frac{\eps}{3}$. By a similar composition but at different scales,
	we get $\phi_4$.  We also obtain $\phi_2$ from \lemref{dL-inclusion} and
	$\phi_3$ from \lemref{dL-inclusion}.

	We argue that $\phi_2$ induces the desired isomorphism on the fundamental
	groups. By \thmref{Rips-persistence} and
	since~$\eps<\frac{\rho}{\delta(15\delta+2)}$, the simplicial map
	$\phi_4\circ\phi_3\circ\phi_2$ induces an isomorphism on all homotopy
	groups. Therefore, $\phi_2$ induces an injective homomorphism on the
	homotopy groups, particularly the fundamental group of $X$.

	We now show that the induced homomorphism is also surjective on the
	fundamental groups by showing that $\phi_2\circ\phi_1$ induces a
	surjection. As observed \thmref{Rips-persistence}, it suffices to show the
	surjection for the the natural inclusion
	$i:\Ri_{\eps}(S)\inclusion\Ri^\eps_{5\delta\eps}(S)$, because $i$ is
	contiguous to $\phi_2\circ\phi_1$.

	We start with a loop $\eta$ in $\Ri^\eps_{5\delta\eps}(S)$. We can assume that
	$\eta$ is made up of edges (one-simplices) of~$\Ri^\eps_{5\delta\eps}$. Let us
	consider an edge $\sigma=\{a,b\}$ in $\eta$, then we have
	$d_\eps(a,b)\leq5\delta\eps$. By the definition of $d_\eps$, there must be a
	sequence of points $a=x_0,x_1,\cdots,x_k=b$ such that for each $i$, the
	segment $[x_i,x_{i+1}]$ is an edge of $\Ri_\eps(S)$. Moreover, we observe for
	later that the diameter of the whole set $\{x_0,\cdots,x_k\}$ in the~$d_\eps$
	metric is not greater than $5\eps\delta$.

	\begin{figure}[thb]
		\centering \includegraphics[scale=1.2]{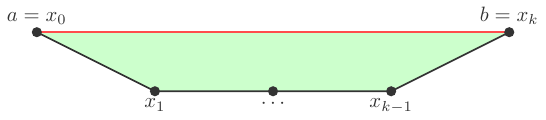}
		\caption{The red one-simplex $[a,b]$ of $\Ri^\eps_{5\delta\eps}(S)$ is shown
			to be pushed off to a path $a=x_0,x_1,\cdots,x_k=b$ in $\Ri_\eps(S)$. All
			the nodes form a simplex (shown in green) in
			$\Ri^\eps_{5\delta\eps}(S)$.}\label{fig:push-off}
	\end{figure}

	Now, we define a loop $\eta'$ in $\Ri_\eps(S)$ by replacing each constituent
	edge $[a,b]$ of $\eta$ by the path joining the points in the sequence
	$a=x_0,x_1,\cdots,x_k=b$ consecutively, as shown in \figref{push-off}.  We
	note that $\eta'$ is indeed a loop in $\Ri_\eps(S)$. We now show that
	$(\phi_2\circ\phi_1)(\eta')$ is homotopic to the loop $\eta$ in
	$\Ri^\eps_{5\delta\eps}(S)$. As observed before, $\{a=x_0,\cdots,x_k=b\}$ is a
	simplex of $\Ri^\eps_{5\delta\eps}(S)$. We can then use a (piece-wise)
	straight line homotopy that maps each edge~$[a,b]$ of $\eta$ to the segment
	$[a=x_0,x_1]\cup\cdots\cup[x_{k-1},x_k=b]$ of $\eta'$. Hence,~$[\eta']$ is, in
	fact, a preimage of~$[\eta]$. This shows, in turn, that $\phi_2$ induces a
	surjective homomorphism on $\pi_1$. This completes the proof.
\end{proof}

%%%%%%%%%%%%%%%%%%%%%%%%%%%
\subsection{Reconstruction of Embedded Graphs}
%%%%%%%%%%%%%%%%%%%%%%%%%%%
\label{subsec:graphs}
We finally turn our attention to the geometric reconstruction of embedded
graphs.  We start with the formal definition of an embedded graph.
\begin{definition}[Embedded Metric Graph]\label{def:metric-graph} An
	\bemph{embedded metric graph} $G$ is a subset of~$\R^N$ that is homeomorphic
	to a one-dimensional simplicial complex, where the induced length metric~$d_L$ is
        the shortest path distance on $G$. For simplicity of exposition,
        we call such $G$ \emph{embedded
        graphs}.
\end{definition}

We note that if $G$ has finitely many
vertices and $b$ is the length of its shortest simple cycle, then the convexity
radius $\rho$ is $\frac{b}{4}$. In this paper, we always assume that $G$ has
finitely many vertices.  We now consider the \emph{shadow} of the Vietoris-Rips
complex $\Ri^\eps_\bullet(S)$, which is defined in \subsecref{fundamental}.
\begin{definition}[Shadow of a Complex]
	Let $A$ be a subset of $\R^N$, and let~$\K$ be an abstract simplicial
	complex whose vertex set is $A$. For each simplex
	$\sigma=\{x_1,x_2,\ldots,x_k\}$ in $\K$, we define its \bemph{shadow},
	denoted~$\sh(\sigma)$, as the convex-hull of the Euclidean point
        set~$\{x_1,x_2,\ldots,x_k\}$.  The shadow of $\K$ in $\R^N$, denoted by
	$\sh(\K)$, is the union of the shadows of all its simplices, i.e., $\sh(\K)
	:= \bigcup\limits_{\sigma\in\K}\sh(\sigma)$.
\end{definition}
We, therefore, have the following natural projection map $p:\mod{\K}\to\sh(\K)$.
In general,~$\sh(\K)$ may not have the same homotopy type as $|\K|$. However, as
proved in~\cite{Chambers2010}, the fundamental group of the Vietoris-Rips
complex of a planar point set is isomorphic to the fundamental group of
its~shadow. In~\cite{adamaszek_homotopy_2016}, the authors further the
understanding of shadows of Euclidean Rips complexes. In the case of planar
subsets and $\K=\Ri^\eps_\bullet(S)$, we prove a similar result now.

\begin{figure}[thb]
	\centering
	\begin{subfigure}[b]{0.45\textwidth}
		\centering
		\includegraphics[scale=.45]{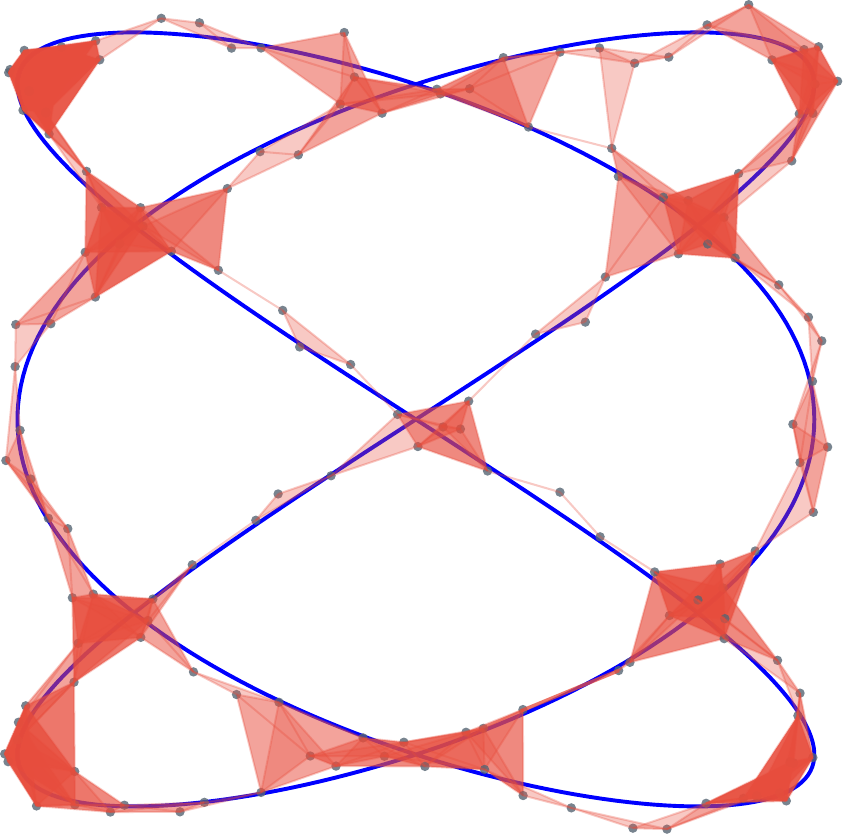}
	\end{subfigure}
	\hfill
	\begin{subfigure}[b]{0.45\textwidth}
		\centering
		\includegraphics[scale=0.45]{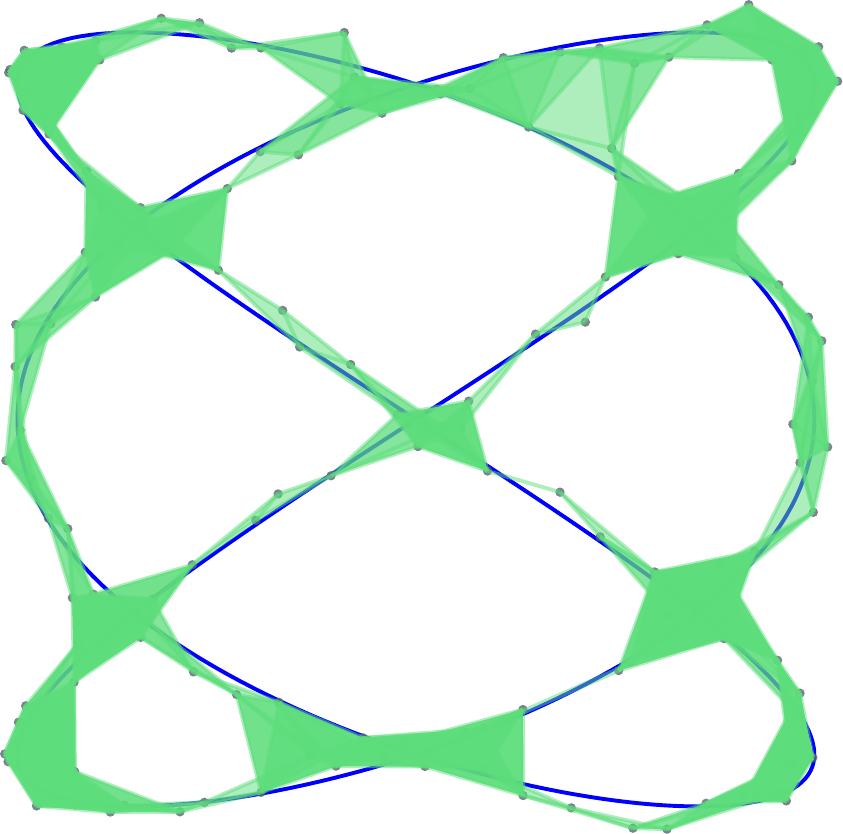}
	\end{subfigure}
	\caption{We implement \algref{graph} on a Lissajous $G$ with $\beta_1(G)=8$.
		On the left, the Euclidean Vietoris-Rips complex $\Ri_\eps(S)$ (in red)
		on an $\eps$-dense sample $S$ of $150$ points fails to
		capture the homotopy type, as its~$\beta_1=9$. On the right, the shadow
		$\widetilde{G}$ (green) of $\Ri^\eps_{5\delta\eps}(S)$ is shown to
		correctly reconstruct $G$. The pictures were generated using the shape
		reconstruction library available on
		\href{https://www.smajhi.com/shape-reconstruction}{\tt
		www.smajhi.com/shape-reconstruction}.}
	\label{fig:recon}
\end{figure}

\begin{lemma}[Shadow]\label{lem:shadow} Let $X$ be a connected planar subspace
	with a positive convexity radius $\rho$ and a finite distortion $\delta$.
	Given $S\subseteq\R^2$ finite and $\eps>0$ such that
	\[
	 d_H(X,S)<\frac{\eps}{3}<\frac{\rho}{\delta(15\delta+2)}.
	 \]
	 Then, the shadow projection
	$p:\mod{\Ri^\eps_{5\eps\delta}(S)}\map\sh(\Ri^\eps_{5\eps\delta}(S))$
	induces isomorphism on the fundamental~groups.
\end{lemma}
\begin{proof}
	From \thmref{fundamental}, we have the following chain of simplicial
	maps:
	\begin{equation*}
		\Ri_{\eps}(S)\map[\phi_1]
		\Ri^L_{5\eps\delta /3}(X)\map[\phi_2]
		\Ri^\eps_{5\delta\eps}(S)\inclusion[\phi_3]
		\Ri_{5\delta\eps}(S)\map[\phi_4]
		\Ri^L_{\delta(15\delta+2)\eps/3}(X).
	\end{equation*}
	We have shown that $\phi_2$ induces an isomorphism on $\pi_1$. As we have
	also noted that $(\phi_4\circ\phi_3\circ\phi_2)$ induces an isomorphism on
	all homotopy groups. So, we conclude first that $\phi_3$ induces an
	injective homomorphism on $\pi_1$ .

	Now, we consider the following commutative diagram:
	\begin{equation}\label{eqn:diag1}
		\begin{tikzpicture} [baseline=(current  bounding  box.center)]
			\node (k1) at (-4,0) {$\Ri_{\eps}(S)$};
			\node (k2) at (-0,0) {$\Ri^\eps_{5\delta\eps}(S)$};
			\node (k3) at (4,0) {$\Ri_{5\delta\eps}(S)$};
			\node (l2) at (0,-2) {$\sh(\Ri^\eps_{5\delta\eps}(S))$};
			\node (l3) at (4,-2) {$\sh(\Ri_{5\delta\eps}(S))$};
			\draw[rinclusion] (k1) to node[auto] {$i$} (k2);
			\draw[rinclusion] (k2) to node[auto] {$\phi_3$} (k3);
			\draw[rinclusion] (l2) to node[auto] {$j_2$} (l3);
			\draw[map] (k2) to node[auto] {$p$} (l2);
			\draw[map] (k3) to node[auto] {$\widetilde p$} (l3);
		\end{tikzpicture}
	\end{equation}
	where $i$ is contiguous to the composition $(\phi_2\circ\phi_1)$, and
	$p,\widetilde p$ are the natural (shadow) projections.

	We show that the induced map $p_*$ is an isomorphism on the fundamental
	groups. From the commutativity of the diagram \eqnref{diag1}, we note that
	$p_*$ is an injection on $\pi_1$, since ${\phi_3}_*$ is injective and
	$\widetilde{p}_*$ is also injective on $\pi_1$ by \cite{Chambers2010}. For
	surjectivity, we follow the same lifting argument presented in
	\cite{Chambers2010}.

	%% We now show that $p_*$ is also surjective on $\pi_1$. Let us start with a
	%loop % $\eta$ in $\sh(\Ri^\eps_{5\delta\eps}(S))$. Since $j$ is the
	%inclusion, % $j_2(\eta)$ is a loop in $\sh(\Ri^\eps_{5\delta\eps}(S))$.
	%Following the % lifting argument as presented in \cite{Chambers2010}, we can
	%lift $j_2(\eta)$. % We note here that $\phi_3$ is a simplicial inclusion.
	%Hence, $j_2(\eta)$ can % eventually be lifted to a path $\widetilde \eta$ in
	%% $\Ri^\eps_{5\delta\eps}(S)$. Now we follow the argument in %
	%\thmref{fundamental} to have a path $\widetilde{\widetilde\eta}$ in
	%$\Ri_\eps$ % such that
	%$\left[pi(\widetilde{\widetilde\eta}\right]=\left[\widetilde \eta]$
\end{proof}

As a consequence of \lemref{shadow}, we finally present our main geometric
reconstruction result.

%\PlanarSpaceReconstruction
\begin{theorem}[Geometric Reconstruction of Planar
	Subspaces]\label{thm:planar-reconstruction} Let $X$ be a connected geodesic
	subspace of $\R^2$ with a positive convexity radius $\rho$ and a finite
	distortion $\delta$, which has the homotopy type of a finite simplicial
	complex. Let $S\subseteq\R^2$ be finite, and $\eps>0$ be such that
	\begin{equation}\label{eq:d_H(X,S)-planar}
	d_H(X,S)<\frac{\eps}{3}<\frac{\rho}{\delta(15\delta+2)}.
	\end{equation}
	Then, the shadow complex $\widetilde{X}=\sh(\Ri^\eps_{5\eps\delta}(S))$ of
	$\Ri^\eps_{5\eps\delta}(S)$ has the homotopy type of $X$. Moreover,
	\begin{equation}\label{eq:d_H(X,sh)}
	d_H(X,\widetilde{X})<\left(5\delta+\frac{1}{3}\right)\eps.
	\end{equation}
\end{theorem}

\begin{proof}
	By \lemref{shadow}, the shadow
	$\widetilde{X}=\sh(\Ri^\eps_{5\eps\delta}(S))$ and $X$  have isomorphic
	fundamental groups, via the map $p$ of diagram \eqref{eqn:diag1}. Note that,
	by assumption, both $\sh(\Ri^\eps_{5\eps\delta}(S))$ and~$X$ have a homotopy
	type of a finite wedge of circles and therefore trivial higher homotopy
	groups. By the Whitehead's theorem \cite{HATCH}, applied to the map $p$, we
	conclude that $p$ is a homotopy equivalence.

	For statement \eqref{eq:d_H(X,sh)}, we note that for any finite vertex set
	$\sigma\subseteq S$ with $\operatorname{diam}(\sigma)< 5\delta\eps$ we have
	$\sigma\subseteq\sh(\sigma)$ and $d_H(\sigma,\sh(\sigma))\leq
	\operatorname{diam}(\sigma)$.  As a consequence, $d_H(\widetilde{X},S)\leq
	5\delta\eps$. By the triangle inequality, we conclude the result.
\end{proof}

%\PlanarGraphReconstruction
\begin{corollary}[Geometric Reconstruction of Embedded Graphs]
%	{corollary}{PlanarGraphReconstruction}
	\label{cor:graphs-reconstruction} Let $G$ be a finite, connected embedded
	graph in~$\R^2$. Let $b$ be the length of the shortest simple cycle of $G$,
	and let $\delta$ be its distortion. Let $S\subseteq\R^2$ be finite
	and~$\eps>0$ be such that
	\[
	 d_H(G,S)<\frac{\eps}{3}<\frac{b}{4\delta(15\delta+2)}.
	\]
	Then, the shadow of $\widetilde{G}=\sh(\Ri^\eps_{5\eps\delta}(S))$ has the
	same homotopy type as $G$ and \eqref{eq:d_H(X,sh)} holds for $X=G$ and
	$\widetilde{X}=\widetilde{G}$.
\end{corollary}

\begin{proof}
	It suffices to note that the convexity radius of $G$ is $\frac{b}{4}$ and
	apply \thmref{planar-reconstruction}. \\
\end{proof}

Based on \corref{graphs-reconstruction}, we devise \algref{graph} for
the geometric reconstruction of (planar) embedded graphs. For a demonstration,
see \figref{recon}.
\begin{algorithm}[htbp]
	\caption{Graph Reconstruction Algorithm \label{alg:graph}}
	\begin{algorithmic}[1]
		\REQUIRE Finite sample $S\subseteq\R^2$, $\eps>0, \delta$, and $b$

		\ENSURE $d_H(\widetilde{G},S)<\frac{\eps}{3}<\frac{b}{4\delta(15\delta+2)}$

		\STATE Initialize $\widetilde{G}\leftarrow\emptyset$\;

		\STATE Compute the one-skeleton of $\Ri_\eps(S)$\;

		\STATE Compute $(S,d_\eps)$\;

		\FORALL{$\{a,b,c\}\in S$}

		\IF{$d_\eps(a,b)\leq5\delta\eps$ \AND $d_\eps(b,c)\leq5\delta\eps$ \AND
			$d_\eps(c,a)\leq5\delta\eps$}

		\STATE $\widetilde{G}\leftarrow\widetilde{G}~\cup$
		\tt{CONVEX-HULL}($\{a,b,c\}$)

		\ENDIF

		\ENDFOR

		\RETURN $\widetilde G$

	\end{algorithmic}
\end{algorithm}

\section{Discussion}
In this paper, we successfully reconstruct homology/homotopy groups of
general geodesic spaces. We also reconstruct the geometry of embedded
graphs. Currently, the output of such geometric reconstruction is a thick region
around the hidden graph; see \figref{recon}. One can consider a post-processing
step to prune the output shadow $\widetilde{G}$ in order to output an embedded
graph that is isomorphic to the hidden graph~$G$.  A natural extension of our
work is to consider the geometric reconstruction of higher-dimensional
simplicial complexes. Unlike the graphs, such a space may have non-trivial
higher homotopy groups. The reconstruction result remains, therefore, an object
of future work.

On the other hand, we also note that both approaches are not performing
well when we deform $X$, e.g., by ``pinching'' a pair of points in $X$, i.e.,
deforming $X$ to bring these points~$\epsilon$--close in the extrinsic Euclidean
distance but with bounded intrinsic distance. Creating such an $\epsilon$--pinch
generally results in a  small $\wfs$ as well as large distortion of the
resulting submanifold.

Based on these considerations, we conjecture that there should be a stability
result within an appropriate class of geodesics subspaces of $\R^N$, saying that
a fixed sample $S$ satisfying assumptions of \thmref{Rips-persistence} and
\thmref{Cech-persistence}, statements \eqref{eq:Rips-persistence} and
\eqref{eq:Cech-persistence} should be valid not only for a given $X$ but also
for any $\epsilon$--close perturbation within the class. We will address this
claim in the forthcoming work.
%with a finite distortion and convexity radius, which should loosely be stated
%as follows: for a fixed sample $S$ of a subspace $X$, and an $\epsilon$--
%perturbed\footnote{e.g. via a bi--Lipshitz map of $\R^N$ with a constant
%$1+\epsilon$} subspace $X_\epsilon$. 

\section*{Acknowledgements}

BTF was supported by grants CCF-1618605, DMS 1854336, and CCF-2046730 from the National Science Foundation.
SM and CW were supported by grant CCF-1618469 from the National Science
Foundation.

\bibliographystyle{ws-ijcga}
\bibliography{reconstruction}

\begin{thebibliography}{10}
\newcommand{\enquote}[1]{#1}

\bibitem{amenta1998crust}
N.~Amenta, M.~Bern and D.~Eppstein, \enquote{The crust and the
  $\beta$-skeleton: Combinatorial curve reconstruction}, \emph{Graphical models
  and image processing} \textbf{60} (1998) 125.

\bibitem{Dey:2006:CSR:1196751}
T.~K. Dey, \emph{Curve and Surface Reconstruction: Algorithms with Mathematical
  Analysis (Cambridge Monographs on Applied and Computational Mathematics)}
  (Cambridge University Press, New York, NY, USA, 2006).

\bibitem{SMALE}
P.~Niyogi, S.~Smale and S.~Weinberger, \enquote{Finding the homology of
  submanifolds with high confidence from random samples}, \emph{Discrete And
  Computational Geometry} \textbf{39. 1-3} (2008) 419.

\bibitem{CHAZALSTAB}
F.~Chazal and A.~Lieutier, \enquote{Stability and computation of topological
  invariants of solids in $\mathbb{R}^n$}, \emph{Discrete \& Computational
  Geometry} \textbf{37} (2007) 601.

\bibitem{co-tpbr-08}
F.~Chazal and S.~Y. Oudot, \enquote{Towards persistence-based reconstruction in
  {Euclidean} spaces}, in \emph{Proc. 24th ACM Sympos. Comput. Geom.} (2008),
  pp. 232--241.

\bibitem{chazal2009sampling}
F.~Chazal, D.~Cohen-Steiner and A.~Lieutier, \enquote{A sampling theory for
  compact sets in {E}uclidean space}, \emph{Discrete \& Computational Geometry}
  \textbf{41} (2009) 461.

\bibitem{latschev_2001}
J.~Latschev, \enquote{Vietoris-{Rips} complexes of metric spaces near a closed
  {Riemannian} manifold}, \emph{Archiv der Mathematik} \textbf{77} (2001) 522.

\bibitem{Chazal2015}
F.~Chazal, R.~Huang and J.~Sun, \enquote{Gromov--{H}ausdorff approximation of
  filamentary structures using {R}eeb-type graphs}, \emph{Discrete {\&}
  Computational Geometry} \textbf{53} (2015) 621.

\bibitem{DeSilva:2004}
V.~De~Silva and G.~Carlsson, \enquote{Topological estimation using witness
  complexes}, in \emph{Proceedings of the First Eurographics Conference on
  Point-Based Graphics} (Eurographics Association, Aire-la-Ville, Switzerland,
  Switzerland, 2004), SPBG'04, pp. 157--166.

\bibitem{ATTALI2013448}
D.~Attali, A.~Lieutier and D.~Salinas, \enquote{Vietoris-{R}ips complexes also
  provide topologically correct reconstructions of sampled shapes},
  \emph{Computational Geometry} \textbf{46} (2013) 448 , 27th Annual Symposium
  on Computational Geometry (SoCG 2011).

\bibitem{Adams_2019}
H.~Adams and J.~Mirth, \enquote{Metric thickenings of {E}uclidean
  submanifolds}, \emph{Topology and its Applications} \textbf{254} (2019) 69.

\bibitem{kim2020homotopy}
J.~Kim, J.~Shin, F.~Chazal, A.~Rinaldo and L.~Wasserman, \enquote{Homotopy
  reconstruction via the \v{C}ech complex and the {V}ietoris-{R}ips complex},
  in \emph{36th International Symposium on Computational Geometry (SoCG 2020)}
  (Schloss Dagstuhl-Leibniz-Zentrum f{\"u}r Informatik, 2020).

\bibitem{aanjaneya2012metric}
M.~Aanjaneya, F.~Chazal, D.~Chen, M.~Glisse, L.~Guibas and D.~Morozov,
  \enquote{Metric graph reconstruction from noisy data}, \emph{International
  Journal of Computational Geometry \& Applications} \textbf{22} (2012) 305.

\bibitem{lecci2014statistical}
F.~Lecci, A.~Rinaldo and L.~A. Wasserman, \enquote{Statistical analysis of
  metric graph reconstruction.}, \emph{Journal of Machine Learning Research}
  \textbf{15} (2014) 3425.

\bibitem{Kim2019}
J.~Kim, J.~Shin, F.~Chazal, A.~Rinaldo and L.~Wasserman, \enquote{Homotopy
  reconstruction via the cech complex and the vietoris-rips complex},  .

\bibitem{adamaszek_homotopy_2016}
M.~Adamaszek, F.~Frick and A.~Vakili, \enquote{On homotopy types of {E}uclidean
  {R}ips complexes}, \emph{Discrete Comput. Geom.} \textbf{58} (2017) 526.

\bibitem{Borsuk:1967}
K.~Borsuk, \emph{Theory of retracts}, Monografie Matematyczne, Tom 44
  (Pa\'{n}stwowe Wydawnictwo Naukowe, Warsaw, 1967).

\bibitem{Dold:1995}
A.~Dold, \emph{Lectures on algebraic topology}, Classics in Mathematics
  (Springer-Verlag, Berlin, 1995), reprint of the 1972 edition.

\bibitem{hausmann_1995}
J.-C. Hausmann, \enquote{On the {Vietoris-Rips} {Complexes} and a {Cohomology}
  {Theory} for {Metric} {Spaces}}, in \emph{Prospects in {Topology}
  ({AM}-138)}, ed. F.~Quinn (Princeton University Press, 1995), Proceedings of
  a {Conference} in {Honor} of {William} {Browder}. ({AM}-138), pp. 175--188.

\bibitem{Lecci:2014:SAM:2627435.2697074}
F.~Lecci, A.~Rinaldo and L.~Wasserman, \enquote{Statistical analysis of metric
  graph reconstruction}, \emph{J. Mach. Learn. Res.} \textbf{15} (2014) 3425.

\bibitem{burago2001course}
D.~Burago, Y.~Burago and S.~Ivanov, \emph{A Course in Metric Geometry}, Crm
  Proceedings \& Lecture Notes (American Mathematical Society, 2001).

\bibitem{gromov1999metric}
M.~Gromov, J.~Lafontaine and P.~Pansu, \emph{Metric Structures for Riemannian
  and Non-Riemannian Spaces}, Progress in Mathematics - Birkh{\"a}user
  (Birkh{\"a}user, 1999).

\bibitem{Kozlov-book}
D.~N. Kozlov, \emph{Combinatorial algebraic topology}, number v. 21 in
  Algorithms and computation in mathematics (Springer).

\bibitem{MUNK}
J.~R. Munkres, \emph{Elements Of Algebraic Topology} (Addison-Wesley Publishing
  Company, 1996), {S}econd edition.

\bibitem{spanier1994algebraic}
E.~H. Spanier, \emph{Algebraic topology}, volume~55 (Springer Science \&
  Business Media, 1994).

\bibitem{gromov1978}
M.~Gromov, \enquote{Homotopical effects of dilatation}, \emph{J. Differential
  Geom.} \textbf{13} (1978) 303.

\bibitem{gromov1983}
M.~Gromov, \enquote{Filling {R}iemannian manifolds}, \emph{J. Differential
  Geom.} \textbf{18} (1983) 1.

\bibitem{Sullivan:2008}
J.~M. Sullivan, \emph{Discrete Differential Geometry} (Birkh{\"a}user Verlag
  AG, Basel, Switzerland, 2008), ch.~Curves of Finite Total Curvature, pp.
  137--161, edited by A.~Bobenko, P.~Schr\"oder, J.M.~Sullivan, and
  G.M.~Ziegler.

\bibitem{Alexandroff1928}
P.~Alexandroff, \enquote{{\"U}ber den allgemeinen {D}imensionsbegriff und seine
  {B}eziehungen zur elementaren geometrischen {A}nschauung},
  \emph{Mathematische Annalen} \textbf{98} (1928) 617.

\bibitem{munkres2000topology}
J.~R. Munkres, \emph{Topology}, Featured Titles for Topology Series (Prentice
  Hall, Incorporated, 2000).

\bibitem{Chambers2010}
E.~W. Chambers, V.~de~Silva, J.~Erickson and R.~Ghrist,
  \enquote{{V}ietoris--{R}ips complexes of planar point sets}, \emph{Discrete
  {\&} Computational Geometry} \textbf{44} (2010) 75.

\bibitem{HATCH}
A.~Hatcher, \emph{Algebraic Topology} (Cambridge University Press, 2002),
  {F}irst edition.

\end{thebibliography}

\end{document}